\documentclass[11pt,reqno]{amsproc}

\usepackage[margin=1.0in]{geometry}
\usepackage[english]{babel}
\usepackage[utf8]{inputenc}
\usepackage{ulem}
\usepackage{amsmath,amssymb,amsthm, comment, mathtools}
\usepackage{hyperref,mathrsfs,xcolor}
\usepackage{enumitem}
\usepackage{latexsym}
\usepackage{amsthm}
\usepackage{amssymb,amsmath}
\usepackage{color}
\usepackage{epstopdf}
\usepackage{hyperref}
\usepackage{multirow}
\usepackage{mathtools}
\usepackage{esint}
\usepackage{tcolorbox,xcolor}

\usepackage{tikz}
\usetikzlibrary{patterns}
\usetikzlibrary{fpu}
\usetikzlibrary{plotmarks}
\usetikzlibrary{decorations.markings}
\usepackage{pgfplots}
\usetikzlibrary{intersections, pgfplots.fillbetween}

\mathtoolsset{showonlyrefs}

\def\lw{\left}
\def\rw{\right}
\def\bega{\begin{aligned}}
\def\enda{\end{aligned}}

\newcommand{\la}{\label}

\newcommand{\be}{\begin{equation}}
\newcommand{\ee}{\end{equation}}
\newcommand{\bea}{\begin{eqnarray}}
\newcommand{\eea}{\end{eqnarray}}

\newcommand{\ve}{{\varepsilon}}

\newcommand{\rmd}{{\rm d}}

\newcommand{\ol}[1]{\mkern 1.5mu\overline{\mkern-1.5mu#1\mkern-1.5mu}\mkern 1.5mu}

\def\wc{\rightharpoonup}

\def\w{{\omega}}

\usepackage{tikz}
\usetikzlibrary{patterns}
\usetikzlibrary{fpu}
\usetikzlibrary{plotmarks}
\usetikzlibrary{decorations.markings}
\usepackage{pgfplots}
\usetikzlibrary{intersections, pgfplots.fillbetween}

\def\beq{\begin{equation}}
\def\eeq{\end{equation}}

\newcommand{\bq}{\begin{equation}}
\newcommand{\eq}{\end{equation}}
\newcommand{\bqa}{\begin{eqnarray*}}
\newcommand{\eqa}{\end{eqnarray*}}

\newcommand{\p}{\partial}

\newcommand{\cc}{{\omega_0}}
\newcommand{\ue}{u_\mathsf{sb}}

\def\eps{\epsilon}
\def\pt{\partial}

\def\bega{\begin{aligned}}
\def\enda{\end{aligned}}

\def\lw{\left}
\def\rw{\right}

\def\la{\langle}
\def\ra{\rangle}

\def\bcase{\begin{cases}}
\def\ecase{\end{cases}}

\def\bmx{\begin{bmatrix}}
\def\emx{\end{bmatrix}}

\theoremstyle{plain}	
\newtheorem{thm}{Theorem}
\newtheorem{lem}[thm]{Lemma}
\newtheorem{prop}[thm]{Proposition}

\newtheorem{question}{Question}

\theoremstyle{definition}

\newtheorem{conj}{Conjecture}

\newtheorem*{rem}{Remark}

\theoremstyle{remark}

\let\emph\relax 
\DeclareTextFontCommand{\emph}{\bfseries\em}

\title{\vspace*{-25mm} The Feynman--Lagerstrom criterion for boundary layers}

\author{Theodore D. Drivas} 
\address{ Department of Mathematics, Stony Brook University,
Stony Brook, NY 11794}
\email{tdrivas@math.stonybrook.edu}

\author{Sameer Iyer}\address{ Department of Mathematics, University of California, Davis, CA 95616}
\email{sameer@math.ucdavis.edu}

 \author{Trinh T. Nguyen}
 \address{Department of Mathematics, University of Wisconsin-Madison, WI 53706}
\email{tnguyen67@wisc.edu}
\begin{document}
\date{\today}

\date{\today}

\begin{abstract}
We study the boundary layer theory for slightly viscous stationary flows forced by an imposed slip velocity at the boundary.  According to the theory of Prandtl (1904) and Batchelor (1956), any Euler solution arising in this limit and consisting of a single ``eddy" must have constant vorticity.  Feynman and Lagerstrom (1956) gave a procedure to select the value of this vorticity by demanding a  \textit{necessary} condition for the existence of a periodic Prandtl boundary layer description.    In the case of the disc, the choice -- known to Batchelor (1956) and Wood  (1957) --  is explicit in terms of the slip forcing.  For domains with non-constant curvature, Feynman and Lagerstrom give an approximate formula for the choice which is in fact only implicitly defined and must be determined together with the boundary layer profile. We show that this condition is also sufficient for the existence of a periodic boundary layer described by the Prandtl equations. Due to the quasilinear coupling between the solution and the selected vorticity, we devise a delicate iteration scheme coupled with a high-order energy method that captures and controls the implicit selection mechanism.
\end{abstract}

\maketitle

\vspace{-5mm}
\section{Introduction}

Let $M\subset \mathbb{R}^2$ be a bounded, simply connected domain. Consider the Navier-Stokes equations
   \begin{align}\label{nsb}
   \partial_t u^\nu +u^\nu \cdot \nabla u^\nu &= -\nabla p^\nu + \nu \Delta u^\nu \qquad \text{in} \ M, \\ \label{incom}
   \nabla \cdot u^\nu &=0 \qquad \qquad  \qquad \ \  \quad \text{in} \ M.
   \end{align}
   Motion is excited through the boundary, where stick boundary conditions are supplied
\begin{align}\label{nonpen}
u^\nu\cdot\hat{n} &= 0 \qquad \text{on}\ \partial M,\\ 
  u^\nu\cdot\hat{\tau} &= f \qquad \text{on}\ \partial M, \label{boundaryforce}
\end{align}
where $\hat{n}$ is the unit outer normal vector field on the boundary and $\hat{\tau}=\hat{n}^\perp$, the unit tangent field.
In the above, there is a given autonomous slip velocity $f:\partial M \to \mathbb{R}$, which should be thought of as being generated by motion of the boundary (via so-called ``stick" or ``no slip" boundary conditions -- the fluid velocity on the boundary matches its speed), and is responsible for the generation of complex fluid motions in the bulk.  If the viscosity is large relative to the forcing, then it is easy to see that all solutions converge to a unique steady state as $t\to \infty$ \cite{Tsai}.  However, as viscosity is decreased, one generally expects solutions to develop and retain non-trivial variation in time, perhaps even forever harboring  turbulent behavior. There one general exception to this expectation in a special setting, proved in \S \ref{noturbthm}:
\begin{thm}[Absence of turbulence] \label{absturb}
Let $M=\mathbb{D}$ be the disk of radius $\mathsf{R}$ and  $f = \frac{1}{2} \cc \mathsf{R}$ for any given $\cc\in \mathbb{R}$ be a constant slip on the boundary.   For any distributionally divergence-free $u_0\in L^2$, the unique Leray-Hopf weak solution converges at long time to solid body rotation $\ue =\frac{1}{2}\cc x^\perp$ having vorticity $\omega_0$. In fact,
\be
 \| u(t)- \ue\|_{L^2}\leq  \| u_0- \ue\|_{L^2} e^{-\lambda_1 \nu t}
\ee
where $\lambda_1$ is the first positive eigenvalue of $-\Delta$ with Dirichlet boundary conditions on $\mathbb{D}$.
\end{thm}

\begin{rem}
 The forcing (slip velocity) in Theorem \ref{absturb} can be arbitrarily large and yet for any viscosity Navier-Stokes has a one-point attractor.  This is the analogue of Marchioro's results on the absence of turbulence on the torus with `gravest mode' body forcing \cite{Marchioro86,Marchioro87}.  \end{rem}

Theorem \ref{absturb} highlights a peculiarity of solid body rotation on the disk: if you center a circular basin on fluid on a record player, all motion will eventually be solid body.
A question arises:
\begin{question}
What if the imposed slip is non-constant, or the domain is not a disk?
\end{question}

 \begin{figure}[h!]
 \centering
 \includegraphics[height=1.8in]{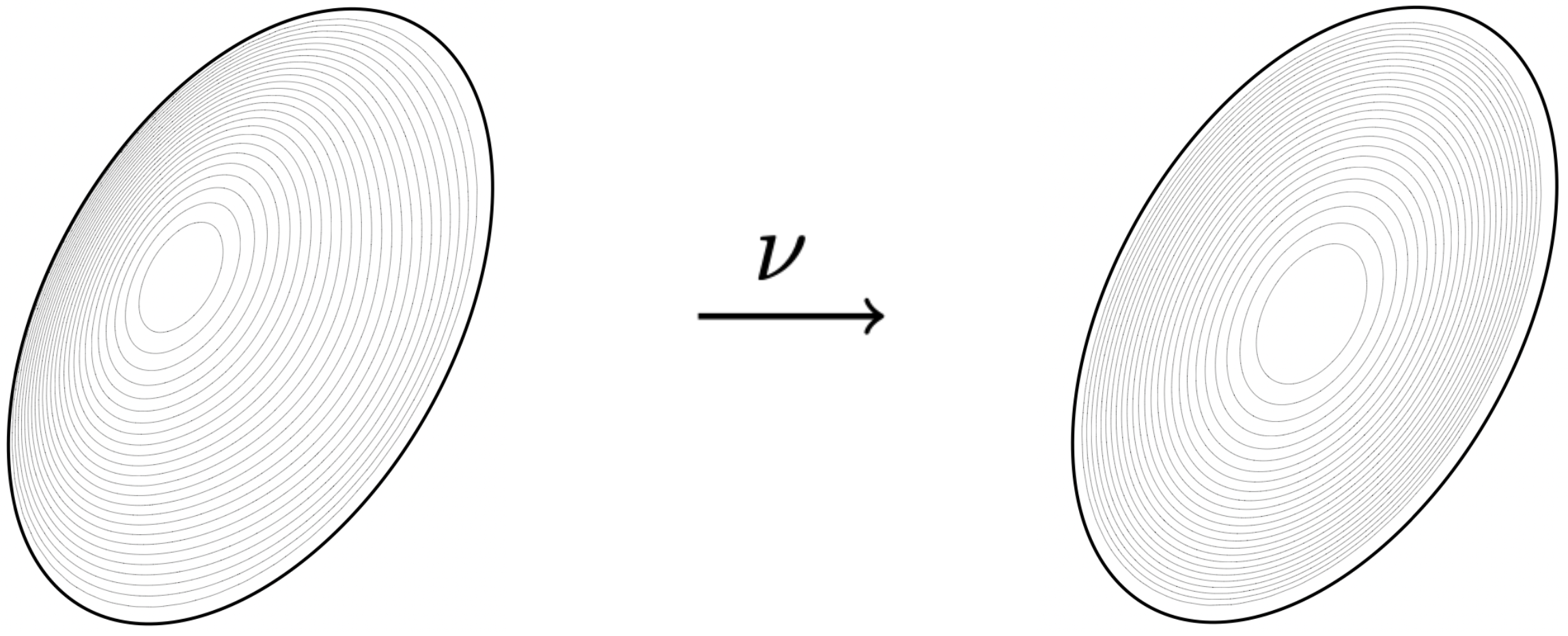} 
 \caption{ Cartoon of streamlines of a steady Navier-Stokes solutions on the ellipse forced by an imposed slip $u^\nu\cdot \tau= f$ on the left. Streamlines of constant vorticity Euler solution $\psi_* =  \frac{a^2}{2(1+a^2)}(\frac{x^2}{a^2}+y^2)$ shown on the right.   
 }
 \end{figure}

As mentioned above, one might expect that if either of these conditions is violated, time dependence generally survives.  However, if the boundary forcing is special this need not be the case. For instance, consider  the velocity field with constant (unit) vorticity on any $M$ 
      \be\label{constvorteuler}
   u_{*} = K_M [1], \qquad K_M:= \nabla^\perp \Delta^{-1}_{ D}.
   \ee
Any such velocity field satisfies both the Euler and Navier-Stokes equations in the bulk.  As such, it is a stationary solution of Euler, and also of Navier-Stokes provided $u_*$ is taken as initial data and it is forced consistently on the boundary: 
$
 f_* =     u_{*}\cdot \tau.
$
Thus, for any domain there is a family of non-trivial time-independent solutions uniformly in $\nu\geq 0$.

 For force sufficiently close to that generated by a stationary Euler solution, asymptotic stability may occur but is a delicate issue.  To begin to understanding these issues, we are interested in the question of the existence of sequence of \textit{stationary} Navier-Stokes solutions approximating an Euler flow in this setting. The   Prandtl--Batchelor theory \cite{Prandtl04,Batchelor56} provides a restriction on the type of stationary Euler solutions that can arises as inviscid limits. Namely, it stipulates that they have constant vorticity within closed streamlines, so-called ``eddies".  See also Childress \cite{Childress,Childress2} and Kim \cite{Kimthesis,Kim98}. The result, proved in \S \ref{appendix}, is:

\begin{thm}[Prandtl-Batchelor Theorem] \label{PBthm}
Let $M\subset \mathbb{R}^2$ be a simply connected domain with smooth boundary.  Let $\psi:M\to \mathbb{R}$ be a $W^{4,1+}(M)$ streamfunction of a steady, non-penetrating solution of Euler $u_e$ having a single stagnation point which is non-degenereate in a sense that the period of revolution of a particle  is a differentiable function of the streamline. Suppose $\{u^\nu\}_{\nu>0}$ is a family  satisfying \eqref{nsb}, \eqref{incom} together with
\be\label{conv}
\lim_{\nu\to 0} \|u ^\nu  -u_e \|_{H^{5/2+}(U)} \to 0,
\ee
for all interior open subsets $U\subset M$. 
Then $u_e= \cc u_*$
for a constant $\cc\in \mathbb{R}$ and $u_*$ is \eqref{constvorteuler}.
  \end{thm}
  
   In the above theorem, $M$ can be thought of as a streamline of an Euler solution occupying some larger spatial domain.  If the limiting Euler solution consists of multiple eddies, the above shows that, within each eddy, the vorticity tends to become constant.
The vorticity of the resulting solution would be a staircase landscape separated, perhaps, by vortex sheets.  Such a picture is consistent with the general expectation of the emergence of weak solutions in the inviscid limit on bounded domains \cite{CLNV,CV,DN19}.  We remark that similar selection principles to Theorem \ref{PBthm} appear also in two-dimensional passive scalar problems \cite{RY,NPR}, and in steady heat distribution in three-dimensional integrable magnetic fields \cite{DGG23}.

    If the boundary data is a sufficiently small perturbation of the corresponding slip of an unit vorticity Euler flow $u_*$ on  whole vessel $M$,
         \be\label{formulaf}
 f =  u_*\cdot \tau +  \ve g.
\ee
    then the inviscid limit of steady Navier-Stokes solutions might be expected to consist of just a single eddy having constant vorticity \eqref{constvorteuler}, that is, for some appropriate constant $\cc\in \mathbb{R}$, 
\be\label{nslim}
u^\nu \to \cc u_* \qquad \text{as} \qquad \nu \to 0.
\ee
See Conjecture \ref{conj}.
This naturally leads to  the following question:
\begin{question}
Given boundary data \eqref{nonpen}, \eqref{boundaryforce}, \eqref{formulaf}, how is the limiting vorticity $\cc$ \eqref{nslim} selected?
\end{question}

This question was discussed by Batchelor (1956) \cite{Batchelor56} and  Wood (1957) \cite{Wood57} for disk domains, and the resulting prediction is called the  Batchelor--Wood formula.  This analysis was done independently by Feynman--Lagerstrom (1956) \cite{FL56} who also generalized this formula to domains with non-constant curvature. See also \cite{Lagerstrom,LagerstromCasten}.   The idea is: the vorticity value $\cc$ is fixed by demanding the corresponding Prandtl equation for the boundary layer admits a periodic solution (that the layer exists).
On the disk $M=  \mathbb{D}$ of radius $\mathsf{R}$, this amounts to:
  \be\label{formfc}
\omega_0^2= \frac{1}{ ( \mathsf{R}/2)^2 } \fint_{0}^{2\pi \mathsf{R}} f^2(\theta) \rmd \theta .
  \ee
  This picture has  been rigorously justified by  Kim \cite{Kim0,Kim} for the boundary layer and  recently by Fei, Gao, Lin and Tao  \cite{Lin1} for Navier-Stokes. The latter constructs a sequence of steady Navier-Stokes solutions on the disk forced by \eqref{formulaf}   converging towards this predicted end state.
In the case of a general domain, Feynman--Lagerstrom argued that selecting $\cc$ to ensure a certain periodicity is a necessary condition for the existence of such a layer and therefore for convergence, but did not speak to its sufficiency.  We now review their theory.

 \begin{figure}[h!]
 \centering
 \includegraphics[width=0.5\textwidth]{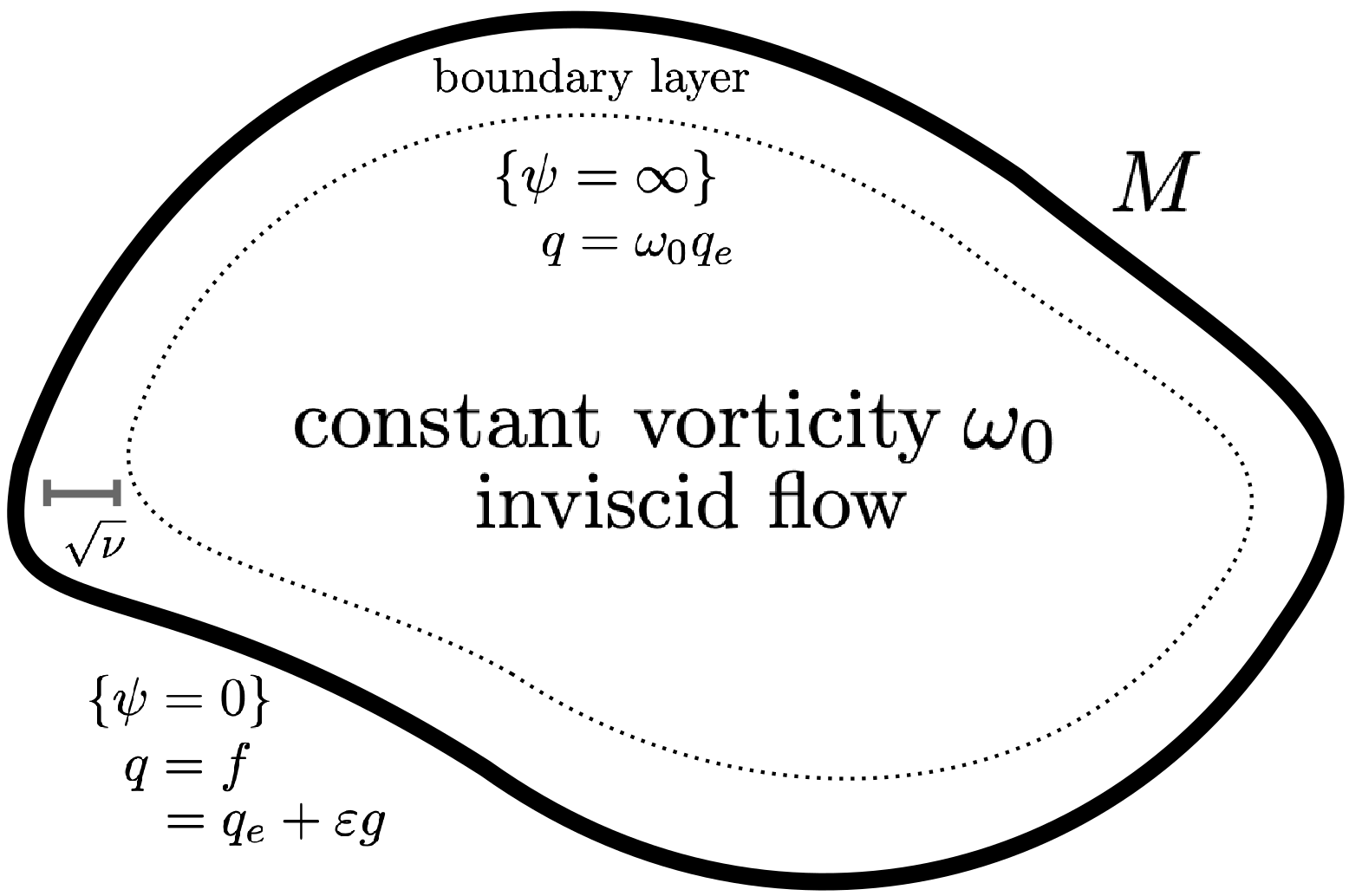} \hspace{-5mm}
  \caption{Boundary layer geometry, depicted (abusing the required asymptotic $\nu\to 0$ for the sake of illustration) in  von Mises coordinates. The level $\psi\in[0,\infty)$ denotes the distance from the boundary in a layer of size $\sqrt{\nu}$, upon rescaling so that this extends indefinetely. See  \S\ref{derivation}.  The {unique} vorticity value $\omega_0$ so that such a boundary layer exists is selected nonlinearly via \eqref{ccond}. It is given approximately by the Feynman-Lagerstrom formulae \eqref{FLcond} or \eqref{FLcond2}.}
 \end{figure}

Recall Prandtl's boundary layer equations written in von Mises coordinates  $(s,\ol{\psi})$, where $s$ is the periodic coordinate on the boundary and $\bar{\psi}:=\psi/\sqrt{\nu}$ is the rescaled streamline coordinate. From hereon, we denote $\bar{\psi}=\psi$ and 
\be
q_e (s):= u_*\cdot \hat{\tau}(\gamma(s))
\ee
where $\gamma : [0,|\partial M|)\to \partial M$ is the arc-length parametrization of the boundary, is   the  tangential slip along the boundary of unit vorticity  Euler solutions  \eqref{constvorteuler}. The Prandtl equations -- which determine an unknown function $q:[0,L)\times \mathbb{R}^+\to \mathbb{R}$ which serves as an approximation of the tangential Navier-Stokes velocity $u^\nu\cdot \hat{\tau}$ in an $O(\sqrt{\nu})$ boundary layer -- are (see \cite{Oleinik}):
\be\label{vmeqn}
\partial_s Q - q\partial_\psi^2 Q= 0, \qquad Q = q^2-\omega_0^2  q_e ^2,
\ee
which is to be satisfied on $(\psi,s)\in [0,\infty)\times [0,L)$ where $L=|\partial M|$ is the length of the boundary. For completeness, we derive these equations in \S\ref{derivation}.
 The solution $q$ must connect Navier-Stokes to Euler: at the boundary ($\psi=0$), the solution $q$ takes the Navier-Stokes data and away from the boundary ($\psi=\infty$), the solution $q$ assumes the Eulerian behavior:
\be
q(0,s)= f(s), \qquad q(\infty,s) =  \cc q_e (s)
\ee
which, for $Q$,  translates to the data
\begin{align}
Q(0,s) &=f^2(s)  - \omega_0^2  q_e ^2(s),
\qquad 
 Q(\infty,s) =0.
 \end{align}
The solution $q$ (or, equivalently, $Q$) of equation \eqref{vmeqn} must be periodic in the $s$ variable (so that the boundary layer closes). Feynman--Lagerstrom   noted that this leads to  a self-consistency condition on $\cc$. We will enforce this in the following way.  First, we rewrite \eqref{vmeqn} as
\be
\partial_s Q -\cc q_e \partial_\psi^2 Q 
=
 \Big(1-  \frac{\cc q_e }{\sqrt{\omega_0^2  q_e ^2+ Q }}\Big)\partial_s Q.
\ee
Integrating the above equation over the boundary, we obtain 
\be
 \partial_\psi^2 \int_0^L  q_e (s) Q(\psi,s) \rmd s = \mathcal{N}^\cc[Q]
\ee
where the nonlinearity is explicitly
\be\label{Nform}
\mathcal{N}^\cc[Q](\psi):= \frac{1}{\cc} \int_0^L  \left(1-  \frac{\cc q_e (s)}{\sqrt{\omega_0^2  q_e ^2(s)+ Q }}\right)\partial_s Q \rmd s.
\ee 
Then, for some scalars $A$ and $B$, we obtain the identity
\be
 \int_0^L   q_e (s)Q(\psi,s) \rmd s 
 = A+ B \psi - \int_\psi^\infty (\psi-y) \mathcal{N}^\cc[Q](y) \rmd y.
\ee
From the boundary conditions, we have that $A=B=0$. We thus obtain the nonlinear condition that at each $\psi\in \mathbb{R}_+$, we have
\be
\int_0^L  q_e (s)Q(\psi,s) \rmd s =  \int_\psi^\infty y \mathcal{N}^\cc[Q](y) \rmd y.
\ee
Evaluating at $\psi=0$, we find a nonlinear, nonlocal condition determining the constant  $\omega_0$:
\be  \label{ccond}
\int_0^L  q_e (s)(f^2(s)  - \omega_0^2  q_e ^2(s)) \rmd s =  \int_0^\infty y \mathcal{N}^\cc[Q](y) \rmd y.
\ee
\begin{rem}[Feynman--Lagerstrom formulae]
Letting $1- \omega_0^2 =:\ve\overline{\omega}_0$ with $\overline{\omega}_0=O(1)$, we anticipate $Q = O(\ve)$.
Since $\frac{1}{\sqrt{ \omega_0^2+x}}-\frac{1}{\sqrt{ \omega_0^2}} = - \frac{x}{2\omega_0^3} + O (x^2)$, we see that $\mathcal{N}^\cc[Q](\psi)=O(Q^2) = O(\ve^2)$.
In fact, we have  $\mathcal{N}^\cc[Q](y)= O(| \partial_s \kappa|\ve^2)$ since it is trivial in the case of the boundary having constant curvature $\kappa:=\hat{\tau}\cdot\nabla \hat{n} \cdot \hat{\tau}$.  Indeed, in this case of $M$ being a disk, it is readily seen that the integrand  in \eqref{Nform} is a total derivative in $s$ and hence $\mathcal{N}^\cc[Q](\psi)\equiv 0$), see the next Remark. Thus, as pointed out by Feynman and Lagerstrom \cite{FL56}, the leading order condition from \eqref{ccond} is 
\be\label{FLcond}
\omega_0^2  =  \frac{\int_0^L  q_e (s)f^2(s) \rmd s}{\int_0^L q_e ^3(s) \rmd s}  + o(|\partial_s \kappa|).
\ee  
This formula is exact (having $\partial_s \kappa=0$) when $M$ is the disk, and generally only for the disk\footnote{Among domains with smooth boundary. For Lipschitz domains, it holds also for regular polygons  \cite{vW}.}. Recalling $ f(s) = q_e (s) +  \ve g(s)$ so that $f^2(s)  - \omega_0^2  q_e ^2(s)=   (1- \omega_0^2 ) q_e ^2(s) + 2\ve g(s)q_e (s) + \ve^2 g^2(s)$, to  leading order in $\ve$ (the deviation of NS data from unit vorticity Euler slip) we have
\be\label{FLcond2}
\omega_0^2 = 1+ 2 \ve   \frac{\int_0^L  q_e ^2(s)g(s) \rmd s}{\int_0^L q_e ^3(s) \rmd s} + O(\ve^2).
\ee  
\end{rem}

\begin{rem}[Wood's formula when $M=\mathbb{D}$]
On the disk of radius $\mathsf{R}$, the constant vorticity solution \eqref{constvorteuler} is solid body rotation $u_*(x)=\frac{1}{2} x^\perp$, so that $q_e =\frac{\mathsf{R}}{2}$ is a constant.  In fact, by Serrin's theorem \cite{Serrin} constraining a domain admitting a solution of $\Delta \psi = 1$ with constant Neumann and Dirichlet data,  the disk is the unique domain for which the solid body Euler solution has constant boundary slip velocity $q_e $.  See also \cite{Kim99}.  On the disk, \eqref{FLcond} without an error term is exact  and, with the circumference  $L=2\pi \mathsf{R}$, agrees with \eqref{formfc} of Wood \cite{Wood57}.  
\end{rem}

In this paper, we rigorously establish this prediction by constructing a periodic boundary layer verifying the Feynman--Lagerstrom condition if  the constant vorticity Euler solution $u_*$ on $M$ defined by \eqref{constvorteuler} has no stagnation points on the boundary. We have

\begin{thm}[Existence of a periodic Prandtl boundary layer] \label{thmbl} Let $M$ be a simply connected domain with $L=|\partial M|$. Denote $\mathbb{T}_L= [0,L)$. 
Let $q_e :\mathbb{T}_L\to \mathbb{R}$ be a smooth, non-vanishing function.  Let  $f(s) := q_e (s)+  \ve g(s)$  with smooth $g: \partial M \to \mathbb{R}$.  For all $\ve$ sufficiently small, depending only the data $(q_e ,g)$, there exists a unique constant $\cc\in \mathbb{R}$ and  function  $q:\mathbb{T}_L\times \mathbb{R}^+\to \mathbb{R}$  such that the pair $(\omega_0, q)$ solves the Prandtl equations  on $\mathbb{T}_L\times \mathbb{R}^+$:
\begin{align} \label{theo:sameer:trinh:1} 
\partial_s Q - q\partial_\psi^2 Q&= 0, \qquad Q = q^2-\omega_0^2  q_e ^2,\\ \label{theo:sameer:trinh:2}
Q(s,0) &=f^2(s)  - \omega_0^2  q_e ^2(s),\\ \label{theo:sameer:trinh:3}
 Q(s,\infty) &=0.
\end{align} 
Moreover, the solution $Q$ lies in the space $X_{2,50}$ defined by \eqref{spaceXkm} and enjoys $\|Q\|_{X_{2,50}} \lesssim \ve$. The selected vorticity $\omega_0$ can be expressed as follows:  there exists a constant $C>0$ so that
\begin{align}\label{constantthm}
\omega_0^2 =  \frac{\int_0^L  q_e (s)f^2(s) \rmd s}{\int_0^L q_e ^3(s) \rmd s} + {\omega}_{{\rm Err}}\qquad \text{where} \qquad |{\omega}_{{\rm Err}}| \leq C \ve^2.
\end{align}
  The sign of $\omega_0$ agrees with that of the background $q_e $ which, in this case, is positive.
  \end{thm}

  This theorem, proved in \S \ref{mainthmproof}, provides the first rigorous confirmation of the Feynman--Lagerstrom formula, and justifies their claim that for $|f-q_e |\ll |f|$ (translating to  $\ve\ll 1$), the leading term in \eqref{FLcond} serves as a good approximation for the selected vorticity.
  
 The constant  $\cc$ satisfies  \eqref{ccond}, and has an explicit component,  determined by $q_e (s)$ and $f(s)$, as well as an implicit component $\overline{\omega}_{{\rm Err}}$ which is smaller amplitude and for which we obtain bounds.  We emphasize that the $\omega_0$ appearing in \eqref{theo:sameer:trinh:2} is \textit{nonlinearity selected} as soon as the domain, $M$, is no longer a disk (for example, $M$ is an ellipse). This requires, at an analytical level, a delicate coupling between the choice of constant, $\omega_0$, and the control of the solution $Q$ in an appropriately chosen norm, which is the main innovation of our work.\footnote{In this respect, the selection mechanism is similar to another arising in fluid dynamics: inviscid damping \cite{BM}.  There, perturbations to certain stable shear flows return to equilibrium in a weak sense, but the which equilibrium they converge to must determined together with the entire time history of the solution.  }

 We anticipate the Prandtl system, which we analyze in this paper, to be stable in the inviscid limit for the full Navier-Stokes system.   Indeed, this is what is proved in \cite{Lin1} when $M$ is a disk and $u_e = x^\perp$. However, as discussed above, in that very special setting the constant $\cc$ can be explicitly determined \eqref{formfc}.  In general, this is not the case and $\cc$ is only implicitly determined by the condition \eqref{ccond} described above, making the inviscid limit more delicate.  Nevertheless, we believe that the nonlinearly determined constant $\omega_0$ will describe, to leading order in viscosity, the selection principle. That is, we believe that Navier-Stokes vorticity $\omega^{\nu}$ should obey an asymptotic expansion
 \begin{align}\label{vortexp}
 \omega^{\nu} = \omega_0 + O(\sqrt{\nu})
 \end{align}
 in the interior of the domain.
 In fact, we issue the following
 \begin{conj}\label{conj}
\textit{ Let $M\subset \mathbb{R}^2$ be any simply connected domain such that the constant vorticity Euler solution $u_*$ on $M$ defined by \eqref{constvorteuler} has a single eddy (streamfunction has a single, non-degenerate, critical point). Suppose that        Navier-Stokes is forced by a slip of the form  \eqref{formulaf}, e.g. $ f =  u_*\cdot \tau +  \ve g$ for some smooth function $g:\partial M \to \mathbb{R}$.  Then, there exists an $\ve_*:=\ve_*(M, g)$ such that for all $\ve<\ve_*$ we have weak convergence in $L^2(M)$ }
 \be
  u^{\nu} \wc \omega_0  u_* \qquad \text{as} \qquad \nu \to 0
 \ee
\textit{along a sequence of steady Navier--Stokes solutions, where $\omega_0:=\omega_0(M,g)$ is \eqref{constantthm} of Thm \ref{thmbl}.}
 \end{conj}
 Of course, stronger convergence can be expected, along with a boundary layer description such as that established by \cite{Lin1} on the disk. The fact that moving boundaries can stabilize the inviscid limit is a well known phenomenon from the work of Guo and Nguyen \cite{GN} and Iyer \cite{Iyer1,Iyer2}.
Verifying the expansion \eqref{vortexp} to prove the above conjecture will require substantially new ideas.  In the context of elliptical domains $M$, this is work in preparation.

Finally we remark that the failure of a boundary layer to exist is indicative of the existence of multiple eddies: constant vorticity regions are separated by internal layers which can be thought of as free boundaries. This can happen either if the constant vorticity solution on that domain has multiple eddies, or if  the given slip data is far from that of a constant vorticity slip (according to our Theorem \ref{thmbl}). Kim \cite{Kim99} showed that if the Navier-Stokes boundary slip is only slightly negative in places, the Prandtl-Batchelor theory still applies to good approximation in the bulk.  For the situation of being far from compatible slip data, see Kim and Childress  \cite{ChildressKim} for an analytical investigation on a rectangle,  Greengard and Kropinski \cite{GK} for a numerical investigation on disk domains, and Henderson, Lopez and Stewart \cite{Henderson} for laboratory experiments.

\section{Derivation of the Prandtl boundary layer equation}
\label{derivation}

In this section, we derive the Prandtl equations for any simply connected domain $M$. Assume that $s:[0,L]\to \partial M$ be the arc-length parametrization of the boundary $\partial M$. For $s\in \mathbb T_L$, let $\tau(s)$ and $n(s)$ be unit the tangential vector to the boundary $\partial M$. There exists $\delta>0$ such that for any $x\in M$ such that ${\rm dist}(x,\pt M)<\delta$, there exists a unique $s\in \mathbb T_L$ and $x(s)\in \pt M$ such that 
\[
{\rm dist}(x,\pt M)=|x-x(s)|.
\]
Moreover, one has the representation 
\[
x(s,z)=x(s)+zn(s),
\]
where $x(s)\in \pt M$ and $z={\rm dist}(x,\pt M)$, see \cite{BNNT}. The map
\[\bega
\{x\in M:\quad 0<{\rm dist}(x,\pt M)<\delta\}&\to \mathbb T_L\times (0,\delta)\\
x&\to (z,s)
\enda 
\] 
is a diffeomorphism.  We also define the following quantities for the domain $M$: 
\[
\gamma(s)=x_1''(s)x_2'(s)-x_1'(s)x_2''(s),\quad J(z,s)=1+z\gamma(s)>0,
\]
where $\gamma$ represents the boundary curvature, and $J$ is a Jacobian for a near-wall mapping used to derived the following form of Navier-Stokes, see \cite{BNNT} and Appendix \ref{derivationNS}.

Now for $x=x(s,z)$, we denote $\tau(s)$ and $n(s)$ to be the tangential and normal vector at $x(s)\in \pt M$ on the boundary. 
Consider the steady Navier-Stokes equations 
\[\bega 
u^\nu\cdot\nabla u^\nu+\nabla p^\nu&=\nu\Delta u^\nu,\\
\nabla\cdot u^\nu&=0,
\enda 
\]
written in the region ${\rm dist}(x,\pt M)<\delta$. We define 
\[\bega
u_\tau(s,z)&=u^\nu(x)\cdot \tau(s)=u^\nu(x(s,z))\cdot\tau(s),\\
u_n(s,z)&=u^\nu(x) \cdot n(s)=u^\nu(x(s,z))\cdot n(s).
\enda
\]
By direct calculation, provided in Appendix \ref{derivationNS}, the Navier-Stokes equations become 
\[
\bega 
&\frac{u_\tau}{J}\pt_s u_\tau+u_n\pt_z u_\tau-\frac{\gamma}{J}u_\tau u_n+\frac 1 J \pt_s p\\
&\quad=\nu\lw\{\frac 1 J \pt_z (J\pt_z u_\tau)+\frac 1 J \pt_s\lw(\frac 1 J \pt_s u_\tau
\rw)-\frac 1 J \pt_s\lw(\frac{\gamma u_n}{J}
\rw)-\frac{\gamma}{J}\lw(\gamma u_\tau+\pt_s u_n
\rw)
\rw\}\\
&\frac{u_\tau}{J} \pt_s u_n+u_n\pt_z u_n -\frac{\gamma}{J}u_\tau^2+\pt_z p\\
&\quad =\nu\lw\{\frac 1 J \pt_z\lw(J\pt_z u_n
\rw)+\frac 1 J \pt_s\lw(\frac 1 J \pt_s u_n
\rw)-\frac 1 J \pt_s\lw(\frac{\gamma u_\tau}{J}
\rw)-\frac{\gamma}{J}\lw(\pt_s u_\tau-\gamma u_n
\rw)
\rw\}\\
&\pt_z u_n+\frac 1 J \pt_s u_\tau-\frac \gamma J u_n=0.
\enda 
\]
\begin{rem}
On the disk with the usual polar coordinates $(\theta,r)\in \mathbb T\times [0,1]$, we have 
\[
\gamma(\theta)=-1,\quad J(z,\theta)=\frac{1}{r},\qquad u_\tau=u_\theta,\quad u_n=u_r.
\]
\end{rem}
Near the boundary, in a layer of width $\sqrt{\nu}$, we anticipate that Navier-Stokes  velocity field $(u_\tau, u_n)$ will look like a small boundary layer correction $(u_\tau^P, v_n^P)$,  to a constant vorticity $(\w_0q_e ,0)$ Euler flow, as discussed in the introduction.  That is,
\[
u_\tau(s,z)\sim \w_0q_e (s)+u_\tau^P\lw(s,Z\rw),\qquad u_n(s,z)\sim \sqrt\nu v_n^P\lw(s,Z\rw),\qquad Z=\frac{z}{\sqrt\nu},
\]
where $\lim_{Z\to \infty}u_\tau^P(s,Z)=0$.  See discussion in Oleinik and  Samokhin \cite{Oleinik}.
Plugging in this ansatz into the Navier-Stokes equations near the boundary, and using the approximation 
\[
\tfrac 1 J=\tfrac {1}{1+z\gamma(s)}=\tfrac 1 {1+\sqrt\nu Z\gamma(s)}\sim 1-\sqrt\nu Z \gamma(s)+O(\nu),
\]
we obtain the equations 
\beq\label{Pr-1-Tr}
\bega 
(\w_0q_e (s)+u_\tau^P)\pt_s(\w_0q_e (s)+u_\tau^P)+v_n^P\pt_Z (q_e (s)+u_\tau^P)+\pt_s p-\pt_Z^2 u_\tau^P=0,\qquad \pt_Z p=0
\enda
\eeq
along with the divergence free condition 
\[
\pt_s(\w_0 q_e (s)+u_\tau^P)+\pt_Z v_n^P=0
\]
Taking $Z\to \infty$ in the equation \eqref{Pr-1-Tr}, we obtain 
\[
\pt_s p=-\w_0^2 q_e (s)q_e '(s)
\]
Replacing the pressure by the above into the equation \eqref{Pr-1-Tr}, we obtain the \textit{Prandtl equations}:
\begin{align}\label{Pr-2-Tr}
(\w_0 q_e (s)+u_\tau^P)\pt_s(\w_0 q_e (s)+u_\tau^P)+v_n^P\pt_Z (\w_0q_e (s)+u_\tau^P)-\w_0^2 q_e (s)\eta_e'(s)-\pt_Z^2 u_\tau^P&=0\\
\pt_s(\w_0q_e (s)+u_\tau^P)+\pt_Z v_n^P&=0.
\end{align}
Define the von Mises variables $(s,\psi)$ such that
\[
\pt_Z \psi(s,Z)=\w_0 q_e (s)+u_\tau^P(s,Z),\qquad -\pt_s \psi(s,Z)=v_n^P(s,Z).
\]
Let $q=q(s,\psi)=\w_0 q_e (s)+u_\tau^P$, the Prandtl equation becomes
\[
q(\pt_s q-v_n^P\pt_\psi q)+qv_n^P \pt_\psi q-\frac 1 2 \pt_s\lw(\w_0^2 q_e (s)^2
\rw)-q\pt_\psi(q\pt_\psi q)=0
\] 
which reduces to 
\[
\pt_s q^2-\pt_s q_e ^2 -q\pt_\psi q^2 =0.
\]
Let $Q=q^2-\w_0^2 q_e ^2$, the above equation becomes \eqref{vmeqn}, namely
\[
\pt_s Q-q\pt_\psi^2 Q=0,\qquad Q=q^2-\w_0^2 q_e ^2.
\]

 \begin{figure}[h!]
 \centering
 \includegraphics[width=0.678\textwidth]{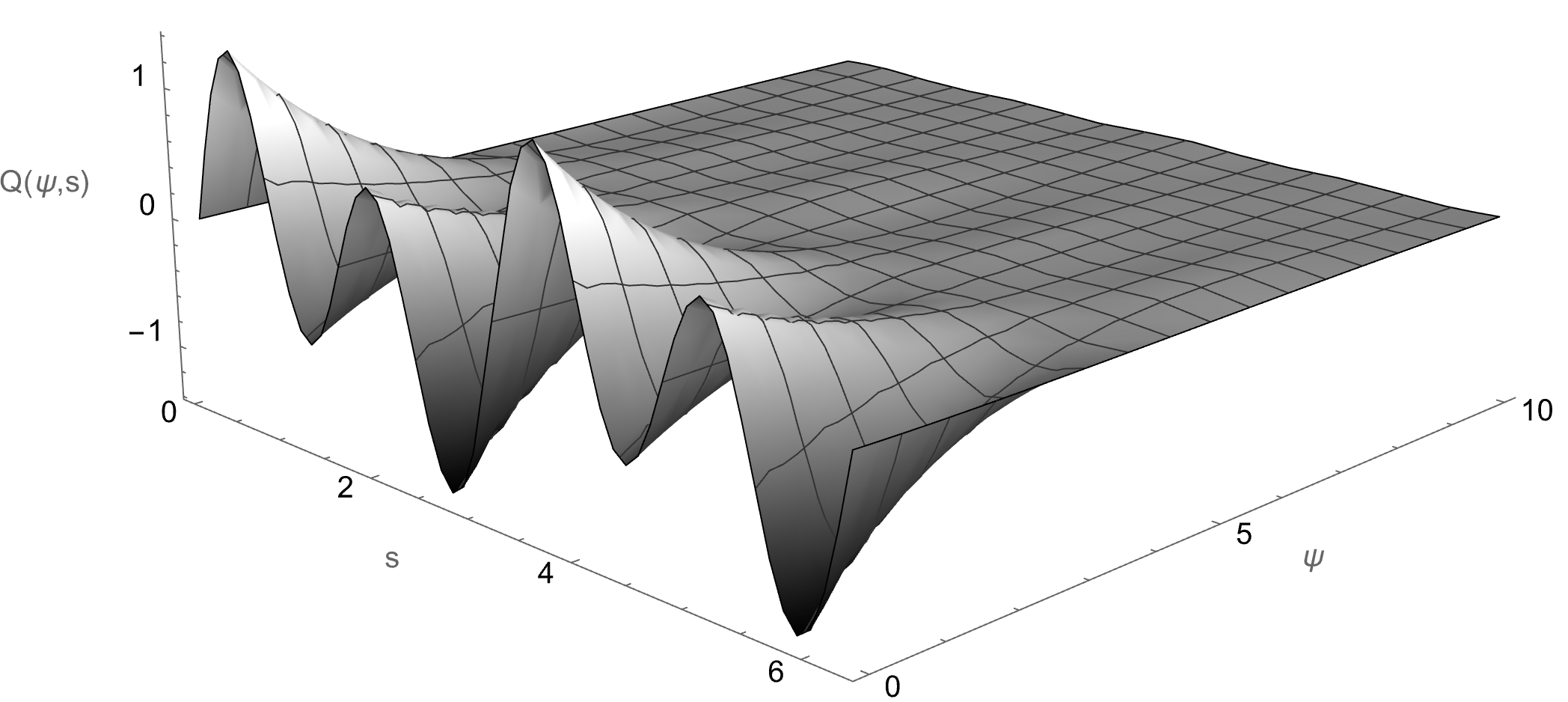} \hspace{-5mm}
  \caption{Numerical solution $Q$ of \eqref{vmeqn} for some representative data $q_e $ and $f$. Thinking of $s$ at ``time", the equation is nonlinear heat sourced at the wall $\psi=0$.  Decay away from the boundary is rapid (exponential) in $\psi$. }
 \end{figure}

\section{Proof of Theorem \ref{absturb}: Absence of turbulence} \label{noturbthm}

Let $v^\nu= u^\nu -  u_e$ be the difference of solutions of Euler and Navier-Stokes, where Navier-Stokes is forced by Euler's slip velocity.  On general domains $M$, it satisfies
   \begin{align}
\partial_t v^\nu+ (v^\nu+ u_e)\cdot\nabla v^\nu +   v^\nu \cdot \nabla   u_e&= -\nabla q +\nu \Delta v^\nu  + \nu \Delta u_e  \ \ \  \ \text{in }\  M,\\
   \nabla \cdot v^\nu &=0 \qquad\qquad\qquad\qquad\qquad  \text{in }\  M,\\
v^\nu\cdot\hat{n} &= 0 \qquad\qquad\qquad\qquad\qquad \text{on}\ \partial M,\\ 
  v^\nu\cdot\hat{\tau} &= 0 \qquad\qquad\qquad\qquad \ \ \ \ \ \ \text{on}\ \partial M. 
\end{align}
Whence the error energy (which holds for Leray-Hopf solutions in dimension two) satisfies
\begin{align}
\frac{1}{2}\frac{\rmd }{\rmd t} \|v^\nu\|_{L^2}^2  &\leq  -  \int_M v^\nu\cdot \nabla u_e \cdot v^\nu \rmd x - \nu \|\nabla v^\nu\|_{L^2}^2 + \nu \int v^\nu \cdot \Delta u_e.
\end{align}
In general, we may bound
\begin{align}
\frac{1}{2}\frac{\rmd }{\rmd t} \|v^\nu\|_{L^2}^2 
&\leq\left(\|\nabla u_e\|_{L^\infty} -  \frac{\nu \lambda_1}{2}\right)  \| v^\nu\|_{L^2}^2 +  \frac{2\nu}{\lambda_1} \| \Delta u_e\|_{L^2}^2,
\end{align}
since $v^\nu|_{\partial \mathbb{D}}=0$ so we may apply the Poincar\'{e} inequality $\lambda_1 \|v^\nu\|_{L^2} \leq \|\nabla v^\nu\|_{L^2}^2$ where $\lambda_1$ is the first positive eigenvalue of $-\Delta_D$ on $M$.   We remark, using the results of \cite[Chapter 7]{Tsai} (which establish uniform bounds on the steady states), a similar energy identity can be used to prove global attraction of the unique steady state for Navier-Stokes forced by imposed slip on any domain, provided viscosity is large enough.

On the disk $M=\mathbb{D}$, if $u_e=\ue = \cc x^\perp$ so that $\nabla \ue = \cc \begin{pmatrix} 0&  -1\\ 1& 0
\end{pmatrix}$ and $\Delta u_e =0$, we have
\be
\int_\mathbb{D} v^\nu\cdot \nabla \ue \cdot v^\nu \rmd x =\int_\mathbb{D} v^\nu\cdot (v^\nu)^\perp \rmd x=0.
\ee
On the disk of radius $\mathsf{R}$, this is $\lambda_1= (j_0/\mathsf{R})^2$ where $j_0$ is the first zero of $J_0$  the Bessel function of the
first kind and order zero). We thus have the stated result.

\begin{rem}
On the ellipse
 $\ue = \cc (-y, \alpha x)$ so that $\nabla \ue = \cc \begin{pmatrix} 0&  -1\\ \alpha& 0\end{pmatrix}$ and $v \cdot \nabla \ue \cdot v= \cc (\alpha-1)v_1v_2$.  It follows that provided
 \be
 \nu>\nu_*:=\cc \lambda_1^{-1}(1-\alpha),
 \ee
 then the solid body rotation solution is the global attractor.  In particular, as the eccentricity of the elliptical domain goes to zero, $\alpha\to 1$ and the critical viscosity $\nu_*$ goes to zero.  Curiously, all flows in this elliptical family are isochronal \cite{Y}, meaning that the period of revolution of a particle does not depend on the particular streamline. As such, the form examples of cut points in group of area preserving diffeomorphisms of those domains, see discussion in \cite{DE22,DM22}.  The lack of differential rotation in the Euler solution may have important consequences for the asymptotic stability and realizability in the inviscid limit.
 \end{rem}

\section{Proof of Theorem \ref{PBthm}: Prandtl--Batchelor Theory}\label{appendix}

First, by   \cite[Lemma 5]{CDG22}, under the stated assumptions we have that
\begin{align}
\Delta \psi &= F(\psi) \qquad \text{on} \ M,\\
 \psi &= c_* \qquad \ \  \  \text{on} \ \partial M
\end{align}
for some $C^1$ function $F:\mathbb{R}\to \mathbb{R}$ and constant $c_*\in \mathbb{R}$.  Suppose without loss of generality that $\{\psi = 0\}$ is the unique critical point in $M$, so that ${\rm rang} (\psi)	= [0,c_*]$.  By the assumption \eqref{conv}, we have the convergence $\psi^\nu \to \psi$ in $H^{7/2+}(U)$ and thus in $C^1(U)$ for all interior open subsets $U\subset M$. It follows that we have convergence of the streamlines (level sets of $\psi^\nu$).  Specifically,
for any $c\in {\rm rang} (\psi)$, the set $\{\psi^\nu = c\}$ is a closed streamline (at least for sufficiently small $\nu:=\nu(c)$) converging to $\{\psi = c\}$.   In what follows, for fixed $c$ we assume $\nu$ is sufficiently small for the above to hold.

 Integrating the Navier-Stokes vorticity balance in the sublevel set  $\{\psi^\nu \leq c\}$ 
\begin{align}\nonumber
0 &=\int_{\{\psi^\nu\leq  c\}}\bigg[ u^\nu \cdot \nabla \omega^\nu  - \nu \Delta \omega^\nu \bigg] \rmd x\\ \label{zeroident}
&=\int_{\{\psi^\nu = c\}}\bigg[ (u^\nu \cdot \hat{n}^\nu)  \omega^\nu  - \nu \hat{n}^\nu\cdot\nabla \omega^\nu \bigg]\rmd \ell =- \nu \int_{\{\psi^\nu = c\}}  \hat{n}^\nu\cdot\nabla \omega^\nu \rmd \ell,
\end{align}
where $\hat{n}^\nu= \nabla \psi^\nu /|\nabla \psi^\nu|$ is the unit normal to streamlines $\{\psi^\nu = c\}$.   Thus
\begin{align}\nonumber
\int_{\{\psi = c\}}  \hat{n}\cdot\nabla \omega \rmd \ell& = \int_{\{\psi = c\}}  \hat{n}\cdot\nabla \omega \rmd \ell -  \int_{\{\psi^\nu = c\}}  \hat{n}^\nu\cdot\nabla \omega^\nu \rmd \ell \\
&= \int_{\{\psi^\nu = c\}}  \hat{n}^\nu\cdot\nabla (\omega-\omega^\nu) \rmd \ell-  \int_\Omega  \Delta F(\psi) \mathbf{1}_{{\{\psi^\nu = c\}}\Delta {\{\psi = c\}} }   \rmd x
\end{align}
 where $A \Delta B$ denotes the symmetric difference between two sets.
Under our assumptions, there exists an open set $O\subset M$ containing the streamline $\{\psi^\nu = c\}$ uniformly in $\nu$. By the trace theorem, 
\be
\int_{\{\psi^\nu = c\}} n^\nu\cdot\nabla (\omega^\nu-\omega) \rmd \ell \lesssim \|n^\nu\|_{H^{1/2+}(O)} \|\omega^\nu-\omega\|_{H^{3/2+}(O)}\lesssim \|\omega^\nu\|_{H^{3/2}(O)} \|\omega^\nu-\omega\|_{H^{3/2+}(O)}.
\ee
Combined with the fact that $\omega= F(\psi)$, we find
\begin{align}\nonumber
F'(c)  \int_{\{\psi = c\}} u \cdot \rmd \ell &=      \int_{\{\psi^\nu = c\}}  \Delta (\omega^\nu- F(\psi))  \rmd x -   \int \Delta F(\psi) \mathbf{1}_{{\{\psi^\nu = c\}}\Delta {\{\psi = c\}} }   \rmd x.
 \end{align}
Thus, for any $\delta>0$, we have the bound
\begin{align}
\left|F'(c)   \int_{\Gamma(c)} u \cdot \rmd s  \right| &\lesssim \  \|\omega^\nu\|_{H^{3/2}(O)} \|\omega^\nu-\omega\|_{H^{3/2+}(O)}  \\
&\qquad + \|  \Delta F(\psi) \|_{L^{1+\delta} (M)} \left({\rm Area} \left({\{\psi^\nu = c\}}\Delta {\{\psi = c\}} \right)\right)^{1/\delta}.
\end{align}
Consequently, using \eqref{conv} and taking the limit of the upper bound, we have
\be
F'(c)  \int_{\Gamma(c)} u \cdot \rmd s =0.
\ee

By our hypotheses that $\psi$ has a single stagnation point $\{\psi =0\}$ in $M$, the circulation $ \oint_{\{\psi = c\}} u \cdot    \rmd \ell \neq 0$ for all $c\neq 0$. Thus, since $F'$ is continuous, we must have that $F'(c)=0$ for all  $c\in {\rm rang} (\psi)$ so that  $F= \cc$ for some $\cc\in \mathbb{R}$.  
 
\section{Proof of Theorem \ref{thmbl}} \label{mainthmproof}

\subsection{Iteration and Bootstraps}
Here we produce a unique solution $(Q,\cc)$ of
\begin{subequations}
\begin{align} \label{Preqn}
\partial_s Q - q\partial_\psi^2 Q&= 0, \\
 Q &:= q^2-\omega_0^2  q_e ^2,\\  \label{data1} 
Q(s,0) &=   \ve\overline{\cc} q_e ^2(s) + 2\ve g(s)q_e (s) + \ve^2 g^2(s)\\ \label{data2} 
 Q(s,\infty) &=0,
 \end{align}
 \end{subequations}
 on $(s,\psi)\in \mathbb{T} \times \mathbb{R}^+$,
 for arbitrary $g:\mathbb{T}\to \mathbb{R}$ and  sufficiently small $\ve:= \ve(g;q_e )$.  Here, $\cc$ is to be determined together with $Q$, and we introduced  $\overline{\cc}\in \mathbb{R}$ (anticipated to be an $O(1)$ quantity as it depen\rmd s on $\ve$) defined by
 \begin{align}
 1-\omega_0^2 =   \ve \overline{\cc}.  
 \end{align}

 To prove this result, it is convenient to rewrite \eqref{Preqn} as 
 \be\label{vmeqnNew}
\partial_s Q -\cc q_e \partial_\psi^2 Q =
 \left(1-  \tfrac{\cc q_e }{\sqrt{\omega_0^2  q_e ^2+ Q }}\right)\partial_s Q.
\ee
We will study of following iteration scheme
\begin{align}
\label{approx}
\partial_s Q_n -\cc_{n-1} q_e \partial_\psi^2 Q_n &=
 \left(1-  \tfrac{\cc_{n-1}  q_e }{\sqrt{\cc_{n-1} ^2 q_e ^2+ Q_{n-1} }}\right)\partial_s Q_{n-1}, \\
Q_n(s,0) &=   (1-  \cc_n^2) q_e ^2(s) + 2\ve g(s)q_e (s) + \ve^2 g^2(s)\\
& = \ve \overline{\cc}_n q_e ^2(s) + 2\ve g(s)q_e (s) + \ve^2 g^2(s) \\
 Q_n(s,\infty) &=0,
 \end{align}
 with   $Q_{-1} = 0$, $\cc_{-1}= 1$. 
 Schematically, we think that $(\cc_{n-1}, Q_{n-1})  \mapsto \cc_n \mapsto Q_n$, that is $\cc_n$ is determined on the onset by a compatibility condition for the linear problem which depen\rmd s on the prior iterate, and $Q_n$ is subsequently solved for $\cc_{n-1}$. Let
\begin{align} \label{hgyut:1}
f(Q, s;\cc)&:=\left(1-  \tfrac{\cc  q_e }{\sqrt{\omega_0^2  q_e ^2+ Q }}\right)\partial_s Q=  \left(1 - \sqrt{1 - \frac{Q}{\omega_0^2  q_e ^2 + Q}}\right) \p_s Q.
\end{align}

In this system $\cc_n$ is chosen to enforce that 
 \begin{align} \label{bar:choice}
  \overline{\cc}_n= \frac{  \fint_{\partial M}  (2 g q_e ^2 + \ve g^2q_e ) \rmd s}{ \fint_{\partial M}  q_e ^3} -  \frac{1}{\cc_{n-1}\ve}\frac{\int_0^\infty y  f(Q_{n-1}(y,s), s;\cc_{n-1}) \rmd y}{ \fint_{\partial M}  q_e ^3}.
 \end{align}
Conceptually, it is clearer to separate out the explicit component of $\overline{\cc}_n$, which is $O(1)$ and independent of $n$, and the smaller amplitude implicit component of $\overline{\cc}_n$ as follows 
\begin{align} 
  \overline{\cc}_n = &\overline{\omega}_{0\ast} +\overline{\cc}_{{\rm Err}, n}, \\  \label{aug:1:1}
\overline{\omega}_{0\ast} := & \frac{  \fint_{\partial M}  (2 g q_e ^2 + \ve g^2q_e ) \rmd s}{ \fint_{\partial M}  q_e ^3},  \\ \label{aug:1:2}
\overline{\cc}_{{\rm Err}, n} := & -   \frac{1}{\cc_{n-1}\ve}\frac{\int_0^\infty y \int_0^{L} f(Q_{n-1}(y,s), s;\cc_{n-1}) \rmd s\rmd y}{ \fint_{\partial M}  q_e ^3}.
\end{align}
With this, we can solve the above equation for $Q_n$.
For $\psi\ge 0$,  we define 
\[
\la \psi\ra=1+\psi
\]
For a function $f=f(s,\psi)$ defined on $\mathbb T_L\times \mathbb R_+$, we define 
\begin{align}
\|f\|_{X_{k,m}}^2&=\sum_{k'=0}^k\sum_{m'=0}^m\lw\{ \|\la \psi\ra^{m'}\pt_s^{k'}f\|_{L^2(\mathbb T_L\times \mathbb R_+)}^2+\|\la \psi\ra^{m'}\pt_s^{k'+1}f\|_{L^2(\mathbb T_L\times \mathbb R_+)}^2\rw\}\\ \label{spaceXkm}
&\quad+\sum_{k'=0}^k\sum_{m'=0}^m\lw\{ \|\la\psi\ra^{m'}\pt_\psi\pt_s^{k'}f\|_{L^2(\mathbb T_L\times \mathbb R_+)}^2+\|\la\psi\ra^{m'}\pt_\psi^2\pt_s^{k'}f
\|_{L^2(\mathbb T_L\times \mathbb R_+)}^2
\rw\}
\end{align}
We will construct the unique solution of the equation \eqref{vmeqnNew} in the space $X_{2,50}$. By the standard Sobolev embedding, we also have 
\[
\|f\|_{L^\infty}\lesssim\|f\|_{L^2}+\|\pt_sf\|_{L^2}+\|\pt_\psi f\|_{L^2}+\|\pt_s\pt_\psi f\|_{L^2}\lesssim\|f\|_{X_{1,0}}
\]
 \begin{rem}[Exponential decay of $Q$ for $\psi\gg1$]
In fact, one can prove existence in a space encoding exponential decay in $\psi$, as should be expect for a (nonlinear) heat equation with data at $\psi=0$.  For simplicity of presentation, we prove only algebraic decay but the requisite modifications involving exponential weights are standard.
 \end{rem}

 Indeed, we have the following result that tells us this iteration is well-defined.
\begin{lem} Let $n \ge 0$. Assume that $Q_{n-1} \in X_{2,50}$, and that \eqref{mainbd:1}--\eqref{mainbd:2} are valid until  index $n -1$. Assume further that $\overline{\cc}_{n}$ is defined according to \eqref{bar:choice}. Then, there exists a unique solution, $Q_n$ to the system \eqref{approx}. 
\end{lem}
\begin{proof}We first of all write the system \eqref{approx} as follows
\begin{align}
\label{approx:simp}
\partial_s Q_n -\cc_{n-1} q_e \partial_\psi^2 Q_n &= F_{n-1}, \\
Q_n(s,0) &=   b_{n-1}(s),\\
 Q_n(s,\infty) &=0,
 \end{align}
 We introduce the variable 
 \begin{align}
 t = J(s), \qquad \frac{dt}{ds} = J'(s) = \cc_{n-1}  q_e (s) 
 \end{align}
 We notice that $\frac{dt}{ds}$ is bounded above and below and hence determines an invertible transformation due to the fact that $\cc_{n-1}  q_e (s) > 0$. We also note that by writing $\cc_{n-1}  q_e (s) = \langle \cc_{n-1}  q_e  \rangle + (\cc_{n-1}  q_e  - \langle \cc_{n-1}  q_e  \rangle)$, we have 
 \begin{align}
 t = \langle\cc_{n-1}  q_e  \rangle s + \int_0^s \cc_{n-1}  (q_e  - \langle q_e  \rangle) ds', 
 \end{align}
 which maps $\mathbb{T}_L$ into $\mathbb{T}_{\cc_{n-1}  \langle q_e  \rangle L}$. We next introduce 
 \begin{align}
 V_n(t, \psi) = V_n(J(s), \psi) = Q_n(s, \psi).
 \end{align}
 This object satisfies the system 
 \begin{align}
 \p_t V_n - \p_\psi^2 V_n = \frac{F_{n-1}}{\cc_{n-1} q_e } =: G_{n-1}.
 \end{align}
 We expand the solution of $V_n$ in a Fourier basis in the $t$ variable as follows
 \begin{align}
 ik \widehat{V}^k_n - \p_\psi^2 \widehat{V}^k_n = \widehat{G}_{n-1}^{k}.
 \end{align}
 The zero mode equation is exactly the Feynman-Lagerstrom formula, \eqref{bar:choice}. For the $k$'th mode, where $k \neq 0$, we write the explicit formula: 
 \[
 \hat V_n^k=e^{-\sqrt{ik}\psi}\hat b_{n-1}^k+\frac{1}{2\sqrt{ik}}\int_0^\infty \lw(e^{-\sqrt{ik}(\psi+\psi')}-e^{-\sqrt{ik}|\psi-\psi'|}\rw)\hat G_{n-1}^kd\psi'
 \]
 where $\sqrt{ik}$ is the complex square root of $ik$ with positive real part.
We now observe that for $G_{n-1} \in X_{2,50}$, the above integrals converge to zero as $\psi \rightarrow \infty$ (when $k \neq 0$). This completes the proof.
\end{proof}
 We define the differences $\triangle  \overline{\cc}_n$ and $\triangle Q_n$ to be 
 \[\bega 
 \triangle  \overline{\cc}_n&=  \overline{\cc}_{n-1}-   \overline{\cc}_{n} \\
 &=   \frac{1}{\cc_{n-1}\ve}\frac{\int_0^\infty y  \int_0^{L} f(Q_{n-1}(y,s), s;\cc_{n-1}) \rmd s\rmd y}{ \fint_{\partial M}  q_e ^3}\\ \label{this:form}
  &\qquad -   \frac{1}{\cc_{n-2}\ve}\frac{\int_0^\infty y \int_0^{L} f(Q_{n-2}(y,s), s;\cc_{n-2}) \rmd s\rmd y}{ \fint_{\partial M}  q_e ^3}.
  \enda
   \]
   \[
   \triangle Q_n=Q_{n}-Q_{n-1}
   \]
Differences in $Q$ obey:
\begin{align}
\partial_s (Q_n-Q_{n-1})  -&\cc_{n-1} q_e \partial_\psi^2 (Q_n- Q_{n-1})  \\
&=
 \left(1-  \tfrac{\cc_{n-1}  q_e }{\sqrt{\cc_{n-1} ^2 q_e ^2+ Q_{n-1} }}\right)\partial_s Q_{n-1}-  \left(1-  \tfrac{\cc_{n-2}  q_e }{\sqrt{\cc_{n-2} ^2 q_e ^2+ Q_{n-2} }}\right)\partial_s Q_{n-2}, \\ \label{diffapprox}
&\qquad +  ( \cc_{n-1} -\cc_{n-2})  \partial_\psi^2 Q_{n-1}, \\
(Q_n-Q_{n-1}) (s,0) &=   \ve ( \overline{\cc}_{n}-  \overline{\cc}_{n-1}) q_e ^2(s),\\
 Q_n(s,\infty) &=0.
 \end{align}
 \begin{rem}[Obtaining the sharper bounds stated in Theorem \ref{thmbl}]
In what follows, we will bootstrap bounds of  $\ve^{1-}$, specifically $\ve^{0.97}$ for $|\overline{\cc}_{{\rm Err}, n}|$ and $\ve^{0.99} $ for $ \|Q_{n} \|_{X_{2, 50}}$, although any power less than 1 would suffice by the same argument given below.  This is not essential, it is to avoid keeping track of large constants for simplicity of the bootstrap argument.  In fact, from these bounds one can deduce a posteriori sharper estimates of the form $  |\overline{\cc}_{{\rm Err}, n}| \le  C_1 \ve $ and $ \|Q_{n} \|_{X_{2, 50}} \leq C_2 \ve$ for some, possibly large, constants $C_1,C_2>0$ by taking the proved bounds on $\cc$ and $Q$, returning to the equation, and performing the estimate again.
 \end{rem}
We will establish the following bounds 
 \begin{align} \label{mainbd:1}
  |\overline{\cc}_{{\rm Err}, n}| &\le \ve^{0.97} \\  \label{mainbd:2}
 \|Q_{n} \|_{X_{2, 50}} &\le \ve^{0.99} \\  \label{mainbd:3}
| \Delta \overline{\cc}_{n}| &\le \ve^{1.97} |\Delta \overline{\cc}_{n-1}| +  \ve^{-0.02} \| \Delta Q_{n-1} \|_{X_{1,4}}\\  \label{mainbd:4}
  \| \Delta Q_n \|_{X_{2, 50}} &\le   \ve^{\frac34} \| \Delta Q_{n-1} \|_{X_{2,50}} + \ve^{\frac12} |\Delta \overline{\cc}_n| + \ve^{\frac32} |\Delta \overline{\cc}_{n-1}|,
 \end{align}
 which immediately imply the main result. The bounds \eqref{mainbd:1} -- \eqref{mainbd:2} show that $(\overline{\cc}_{{\rm Err},n}, Q_n) \in B_{\ve^{1.99}, \ve^{0.99}} \subset \mathbb{R} \times X_{2, 50}$, whereas the bounds \eqref{mainbd:3} -- \eqref{mainbd:4} show that iteration converges to a unique fixed point. A standard fixed point result imply that these bootstrap bounds give the main theorem:
 \begin{proof}[Proof of Theorem \ref{thmbl}] We insert the bound \eqref{mainbd:3} into the second term on the right-hand side of \eqref{mainbd:4} in order to get the following 
 \begin{align}
 | \Delta \overline{\cc}_{n}| &\le \ve^{1.97} |\Delta \overline{\cc}_{n-1}| +  \ve^{-0.02} \| \Delta Q_{n-1} \|_{X_{2,50}} \\
  \| \Delta Q_n \|_{X_{2, 50}} &\le  \ve^{\frac14} \| \Delta Q_{n-1} \|_{X_{2,50}} + \ve^{\frac32} |\Delta \overline{\cc}_{n-1}|.
 \end{align}
 Define $\vec{Y}_n := (\Delta Q_n, \ve \Delta \overline{\cc}_n) \in X_{2,50} \times \mathbb{R}$, endowed with the product norm. Then 
 \begin{align}
 \| \vec{Y}_n \|_{X_{2,50} \times \mathbb{R}} \le \ve^{\frac15}\| \vec{Y}_{n-1} \|_{X_{2,50} \times \mathbb{R}}.
 \end{align}
 It is therefore clear that 
 \begin{align}
  \| \vec{Y}_n \|_{X_{2,50} \times \mathbb{R}} \le &  \ve^{\frac{n}{5}}\| \vec{Y}_{0} \|_{X_{2,50} \times \mathbb{R}} \\
  \le &    \ve^{\frac{n}{5}} (    \| Q_{-1} \|_{X_{2,50}} +   \| Q_0 \|_{X_{2,50}} + \ve |\overline{\cc}_{-1}| +  \ve |\overline{\cc}_0|) \\
  \le & C\ve^{\frac{n}{5}}.
 \end{align}
 This then implies that $(Q_n, \overline{\cc}_{n})$ is a Cauchy sequence in $X_{2,50} \times \mathbb{R}$, and hence converges to a limit, $(Q_{\infty}, \overline{\cc}_{\infty})$. We can therefore pass to the limit in equation \eqref{approx} as well as in \eqref{bar:choice} to conclude that $(Q_\infty, \overline{\cc}_{\infty})$ satisfy the system \eqref{Preqn}. 
 
 We now prove uniqueness. We assume that $(Q_{1}, \cc_{,1})$ and $(Q_2, \cc_{,2})$ are two solutions to \eqref{Preqn} in the space $X_{2,50} \times \mathbb{R}$. We may therefore write an analogous equation \eqref{diffapprox} on $Q_1 - Q_2$ (without the iteration), which reads: 
 \begin{align}
\partial_s (Q_1-Q_{2})  -&\cc_{,1} q_e \partial_\psi^2 (Q_1- Q_{2})  \\
&=
 \left(1-  \tfrac{\cc_{,1}  q_e }{\sqrt{\cc_{,1} ^2 q_e ^2+ Q_{1} }}\right)\partial_s Q_{1}-  \left(1-  \tfrac{\cc_{,2}  q_e }{\sqrt{\cc_{,2} ^2 q_e ^2+ Q_{2} }}\right)\partial_s Q_{2}, \\ \label{dhgyu:1}
&\qquad +  ( \cc_{,1} -\cc_{,2})  \partial_\psi^2 Q_{2}, \\
(Q_1-Q_{2}) (s,0) &=   \ve ( \overline{\cc}_{,1}-  \overline{\cc}_{,2}) q_e ^2(s)\\
(Q_1 - Q_2)(s,\infty) &=0
 \end{align}
as well as the analogue of expression \eqref{this:form} (again without the iteration)
\begin{align}
  \overline{\cc}_{1}-   \overline{\cc}_{2} &=   \frac{1}{\cc_{,2}\ve}\frac{\int_0^\infty y  \int_0^{L} f(Q_{2}(y,s), s;\cc_{,2}) \rmd s\rmd y}{ \fint_{\partial M}  q_e ^3}\\
  &\qquad -   \frac{1}{\cc_{,1}\ve}\frac{\int_0^\infty y \int_0^{L} f(Q_{1}(y,s), s;\cc_{,1}) \rmd s\rmd y}{ \fint_{\partial M}  q_e ^3}.\end{align}
 Re-applying the \textit{a-priori} estimates on these systems results in the following bounds: 
 \begin{align}
  | \overline{\cc}_{,1} - \overline{\cc}_{,2}| \le & \eps^{1.97} | \overline{\cc}_{,1} - \overline{\cc}_{,2}| + \eps^{-0.02} \| Q_1 - Q_2 \|_{X_{2,50}}, \\
 \| Q_1 - Q_2 \|_{X_{2,50}} \le & \ve^{\frac34} \| Q_1 - Q_2 \|_{X_{2,50}} + \ve^{\frac12} | \overline{\cc}_{,1} - \overline{\cc}_{,2}|,
 \end{align} 
 which are the analogues of \eqref{mainbd:3} - \eqref{mainbd:4}. The two bounds above clearly imply that $ \overline{\cc}_{,1} = \overline{\cc}_{,2}$ and $Q_1 = Q_2$. This proves uniqueness. 
\end{proof}

 \subsection{$\overline{\cc}_{{\rm Err}, n}$ estimates} 
 
 Here, we will establish the bootstrap bound \eqref{mainbd:1}. Indeed,  
 \begin{lem} Assume \eqref{mainbd:1} -- \eqref{mainbd:4} are valid until  the index $n - 1$. Then $\overline{\cc}_{{\rm Err},n}$ satisfies: 
 \begin{align} \label{boots:cerr}
 |\overline{\cc}_{{\rm Err}, n}| \le \ve^{0.97}.
 \end{align}
 \end{lem}
 \begin{proof} Recall the expression \eqref{aug:1:2}, after which we estimate as follows 
 \begin{align}
  |\overline{\cc}_{{\rm Err}, n}| &\le \frac{1}{\ve} \left|\int_0^\infty y \int_0^{L} f(Q_{n-1}(y,s), s;\cc_{n-1}) \rmd s\rmd y\right| \\
 & \lesssim \frac{1}{\ve} \left\| \left(1-  \tfrac{\cc_{n-1}  q_e }{\sqrt{\cc_{n-1}^2 q_e ^2+ Q_{n-1} }}\right)\partial_s Q_{n-1} \langle \psi \rangle^4 \right\|_{L^2} \\
& \lesssim \frac{1}{\ve} \left\| \left(1-  \tfrac{\cc_{n-1}  q_e }{\sqrt{\cc_{n-1}^2 q_e ^2+ Q_{n-1} }}\right) \|_{L^\infty} \| \partial_s Q_{n-1} \langle \psi \rangle^4 \right\|_{L^2} \\
 &= \frac{1}{\ve} \left\| 1 - \sqrt{1 - \frac{Q_{n-1}}{\cc_{n-1}^2 q_e ^2 + Q_{n-1}}}\right\|_{L^\infty} \left\| \p_s Q_{n-1} \langle \psi \rangle^4 \right\|_{L^2} \\
 &\lesssim  \frac{1}{\ve} \left\| \frac{Q_{n-1}}{\cc_{n-1}^2 q_e ^2 + Q_{n-1}}\right\|_{L^\infty} \left\| \p_s Q_{n-1} \langle \psi \rangle^4 \right\|_{L^2}  \lesssim  \frac{1}{\ve} \ve^{0.99} \ve^{0.99} <  \ve^{0.97},
 \end{align}
 where we have invoked the bootstrap bound \eqref{mainbd:2} in the final step, as well as the $L^\infty$ estimate 
 \begin{align}
 \left\| \frac{Q_{n-1}}{\cc_{n-1}^2 q_e ^2 + Q_{n-1}}\right\|_{L^\infty} &\lesssim  \frac{1}{\cc_{n-1}^2 q_e ^2 - \| Q_{n-1} \|_{L^\infty}} \| Q_{n-1} \|_{L^\infty} \\
& \lesssim  \frac{1}{\cc_{n-1}^2 q_e ^2 - \| Q_{n-1} \|_{X_{1,0}}} \| Q_{n-1} \|_{X_{1,0}} \\
& \lesssim  \frac{1}{\cc_{n-1}^2 q_e ^2 - C \ve^{0.99}} \ve^{0.99} \lesssim \ve^{0.99}.
 \end{align}
Above, we have used the following Sobolev inequality on $\mathbb{T} \times \mathbb{R}_+$, which reads $\| f \|_{L^\infty} \lesssim \| f \|_{X_{1,0}}$. This Sobolev embedding will be used repeatedly to estimate nonlinear terms.
 \end{proof}
 
  \subsection{$\Delta \overline{\cc}_{n}$ estimates} 
  
  Here we prove the following lemma. 
  \begin{lem} Assume  \eqref{mainbd:1} -- \eqref{mainbd:4} are valid until  the index $n - 1$. Assume \eqref{mainbd:1} and \eqref{mainbd:2} are valid until  index $n$. Then the following bound hol\rmd s 
  \begin{align}
  |\Delta \overline{\cc}_n| \le \eps^{1.97} |\Delta \overline{\cc}_{n-1}| +  \ve^{-0.02} \| \Delta Q_{n-1} \|_{X_{1,4}}
  \end{align}
  \end{lem}
\begin{proof} We use the expression 
 \begin{align*}
\Delta (\overline{\cc}_{n}) &= \overline{\cc}_{n-1}-   \overline{\cc}_{n}\\
& :=   \frac{1}{\cc_{n-1}\ve}\frac{\int_0^\infty y  \int_0^{L} f(Q_{n-1}(y,s), s;\cc_{n-1}) \rmd s\rmd y}{ \fint_{\partial M}  q_e ^3}-   \frac{1}{\cc_{n-2}\ve}\frac{\int_0^\infty y \int_0^{L} f(Q_{n-2}(y,s), s;\cc_{n-2}) \rmd s\rmd y}{ \fint_{\partial M}  q_e ^3} \\
&=   \frac{1}{\cc_{n-1}\ve}\frac{\int_0^\infty y  \int_0^{L} f(Q_{n-1}(y,s), s;\cc_{n-1}) \rmd s\rmd y}{ \fint_{\partial M}  q_e ^3} -  \frac{1}{\cc_{n-2}\ve}\frac{\int_0^\infty y  \int_0^{L} f(Q_{n-1}(y,s), s;\cc_{n-1}) \rmd s\rmd y}{ \fint_{\partial M}  q_e ^3} \\
& \quad+  \frac{1}{\cc_{n-2}\ve}\frac{\int_0^\infty y  \int_0^{L} f(Q_{n-1}(y,s), s;\cc_{n-1}) \rmd s\rmd y}{ \fint_{\partial M}  q_e ^3} - \frac{1}{\cc_{n-2}\ve}\frac{\int_0^\infty y \int_0^{L} f(Q_{n-2}(y,s), s;\cc_{n-2}) \rmd s\rmd y}{ \fint_{\partial M}  q_e ^3} \\
&=  I_1 + I_2. 
 \end{align*}
 To estimate $I_1$, we have 
 \begin{align}
 |I_1| &=  \frac{|\cc_{n-2} - \cc_{n-1}|}{\ve \cc_{n-1} \cc_{n-2}}\frac{|\int_0^\infty y  \int_0^{L} f(Q_{n-1}(y,s), s;\cc_{n-1}) \rmd s\rmd y|}{ \fint_{\partial M}  q_e ^3} \\
& \lesssim  \frac{1}{\eps}|\cc_{n-2} - \cc_{n-1}| \| \la\psi\ra^4 f(Q_{n-1},s; \cc_{n-1}) \|_{L^2} \\
& \lesssim  \frac{1}{\eps}|\cc_{n-2} - \cc_{n-1}| \eps^{0.99} \eps^{0.99} \\
 &= \frac{1}{\eps} |\cc_{n-2}^2 - \cc_{n-1}^2| |{\cc}_{n-2} + {\cc}_{n-1}|^{-1}  \eps^{0.99} \eps^{0.99} \\
 &\lesssim  \frac{1}{\eps} \eps |\overline{\cc}_{n-2} - \overline{\cc}_{n-1}|  \eps^{0.99} \eps^{0.99} \\
& \le  \frac12 \eps^{1.97} |\Delta \overline{\cc}_{n-1}|.
 \end{align}
To estimate $I_2$, we need to use the identity \eqref{hgyut:1} to estimate 
\[\bega 
f(Q, \cc) - f(\overline{Q}, \overline{\cc}) 
&=  \left(1 - \sqrt{1 - \frac{Q}{\omega_0^2  q_e ^2 + Q}}\right) \p_s Q - \left(1 - \sqrt{1 - \frac{\overline{Q}}{\overline{\cc}^2 q_e ^2 + \overline{Q}}}\right) \p_s \overline{Q} \\
&=   \left(1 - \sqrt{1 - \frac{Q}{\omega_0^2  q_e ^2 + Q}}\right) \p_s Q - \left(1 - \sqrt{1 - \frac{Q}{\omega_0^2  q_e ^2 + Q}}\right) \p_s \overline{Q} \\
& \quad+  \left(1 - \sqrt{1 - \frac{Q}{\omega_0^2  q_e ^2 + Q}}\right) \p_s \overline{Q} -  \left(1 - \sqrt{1 - \frac{\overline{Q}}{\overline{\cc}^2 q_e ^2 + \overline{Q}}}\right) \p_s \overline{Q} \\
&= BD_1 + BD_2.
\enda
\]
Clearly, using the inequality $|1 - \sqrt{1-x}| \le |x|$ for $x\le 1$, we have 
\begin{align}
\| BD_1 \langle \psi \rangle^4 \|_{L^2} &\lesssim  \lw\| 1 - \sqrt{1 - \frac{Q}{\omega_0^2  q_e ^2 + Q}}\rw\|_{L^\infty} \| \p_s (Q - \overline{Q}) \langle \psi \rangle^4 \|_{L^2} \\
&\lesssim  \| Q \|_{X_{1,0}} \| Q - \overline{Q} \|_{X_{0,4}}.
\end{align}
and, using the inequality $|\sqrt{1-x} - \sqrt{1-y}| \lesssim |x - y|$ for $x, y \ll 1$, we have 
\begin{align}
\| BD_2 \langle \psi \rangle^4 \|_{L^2} &\lesssim  \| \p_s \overline{Q} \langle \psi \rangle^4 \|_{L^2} \lw\| \frac{Q}{\omega_0^2  q_e ^2 + Q} - \frac{\overline{Q}}{\overline{\cc}^2 q_e ^2 + \overline{Q}}\rw \|_{L^\infty} \\
& \lesssim  \| \p_s \overline{Q} \langle \psi \rangle^4 \|_{L^2} \lw\|\frac{\overline{\cc}^2 Q -\omega_0^2  \overline{Q}}{(\omega_0^2  q_e ^2 + Q)(\overline{\cc}^2 q_e ^2 + \overline{Q})}\rw\|_{L^\infty} \\
&  \lesssim  \| \p_s \overline{Q} \langle \psi \rangle^4 \|_{L^2}\lw \| Q - \overline{Q} \rw\|_{L^\infty} +  \| \p_s \overline{Q} \langle \psi \rangle^4 \|_{L^2} \|Q \|_{L^\infty} |\omega_0^2  - \overline{\cc}^2| \\
  &\lesssim  \| \overline{Q} \|_{X_{0,4}} \| Q - \overline{Q} \|_{X_{1,0}} + \| \overline{Q} \|_{X_{0,4}} \| Q \|_{X_{1,0}} |\omega_0^2  - \overline{\cc}^2|.
\end{align}

Therefore, we have 
\begin{align}
|I_2| &\lesssim  \frac{1}{\eps}  \| Q_{n-1} \|_{X_{1,0}} \| Q_{n-1} - Q_{n-2} \|_{X_{0,4}} + \frac{1}{\eps} \| Q_{n-2} \|_{X_{0,4}} \| Q_{n-1} - Q_{n-2} \|_{X_{1,0}} \\
&\quad+ \frac{1}{\eps} \| Q_{n-2} \|_{X_{0,4}} \| Q_{n-1} \|_{X_{1,0}} | \cc_{n-1}^2 - \cc_{n-2}^2| \\
&\lesssim  \frac{1}{\ve} \eps^{0.99} \| \Delta Q_{n-1} \|_{X_{1,4}} + \frac{1}{\ve} \eps^{0.99} \eps^{0.99} \eps |\overline{\cc}_{n-1} - \overline{\cc}_{n-2}| \\
&\le  \ve^{-0.02} \| \Delta Q_{n-1} \|_{X_{1,4}} + \frac{\ve^{1.97}}{2} |\overline{\cc}_{n-1} - \overline{\cc}_{n-2}|.
\end{align}

Pairing these bounds together, we get the desired result. 
\end{proof}

\subsection{Abstract $Q$ Estimates}
 
 For future use, it turns out we will have a need to develop our estimates on a slightly more abstract system. Therefore, we consider 
\begin{align} \label{real:1}
\p_s Q - \cc q_e  \p_\psi^2 Q = & F + \p_{\psi}^2 G \\ 
Q(s, 0) = & b(s) \\
Q(s, \infty) = & 0. 
\end{align}
We develop a high-order energy method to treat equation \eqref{real:1}. We commute $\p_s^k$ to obtain 
\begin{align} \label{real:2}
\p_s Q^{(k)} - \cc q_e  \p_\psi^2 Q^{(k)} = & F^{(k)} + \p_{\psi}^2 G^{(k)} + A_{\text{comm},k} \\ 
Q^{(k)}(s, 0) = & b^{(k)}(s), \\
Q^{(k)}(s, \infty) = & 0,
\end{align}
where the commutator term 
\begin{align}
A_{\text{comm},k} := \sum_{k' = 0}^{k-1} \binom{k}{k'} \p_s^{k-k'} q_e  \p_\psi^2 Q^{(k')},
\end{align}
and where we adopt the short-hand $f^{(k)}(s, \psi) := \p_s^k f(s, \psi)$ for an abstract function $f(s, \psi)$. 

 \begin{prop} \label{pro:pro:1} Assume that the boundary condition $b(s)$ and the source term $F + \p_{\psi}^2 G$ satisfy the Feynman-Lagerstrom compatibilty condition  \eqref{bar:choice}. Then the solution $Q$ to \eqref{real:1} obeys the following inequality: 
 \begin{align} \label{sam:ad:1}
 \| Q \|_{X_{k,m}} \lesssim \sum_{k' = 0}^k \| \p_s^{k'}  F \langle \psi \rangle^{m+4} \|_{L^2} +\sum_{k' = 0}^k \sum_{j= 0}^2 \| \p_s^{k'} \p_{\psi}^j G \langle \psi \rangle^m \|_{L^2} +\|b\|_{H^{k+1}_s} .
 \end{align}
 \end{prop}
 
 The first task is we lift the boundary condition $b(s)$ by considering the lift function 
 \begin{align}
 L(s, \psi) := L[b](s, \psi) := e^{-\psi} b(s)
 \end{align}
 and consequently 
 \begin{align}
 \mathring{Q} := Q - L[b], 
 \end{align}
 which satisfies the following system 
 \begin{align} \label{hghg:1}
\p_s \mathring{Q} - \cc q_e  \p_\psi^2 \mathring{Q} = & F + \p_{\psi}^2 G + G_{\text{Lift}}\\
\mathring{Q}(s, 0) = & 0 \\
\mathring{Q}(s, \infty) = & 0
\end{align}
where 
\beq\label{G-Lift}
G_{\text{Lift}} := e^{-\psi} ( \cc q_e  b(s) -  b'(s)).
\eeq
We will need to work in higher order norms. Therefore, we present the equations upon commuting $\p_s^k$ to \eqref{hghg:1}, which yield 
 \begin{align} \label{hghg:k}
\p_s \mathring{Q}^{(k)} - \cc q_e  \p_\psi^2 \mathring{Q}^{(k)} = & F^{(k)} + \p_{\psi}^2 G^{(k)} +G_{\text{Lift}}^{(k)}+ C_{\text{comm},k}\\
\mathring{Q}^{(k)}(s, 0) = & 0 \\
\mathring{Q}^{(k)}(s, \infty) = & 0. 
\end{align}
Above, we define the commutator term as follows: 
\beq\label{C-comm}
C_{\text{comm},k} := \cc \bold{1}_{k \ge 1} \sum_{k' =0}^{k-1} \binom{k}{k'}\p_s^{k-k'} q_e  \p_\psi^2 \mathring{Q}^{(k')}.
\eeq

\begin{lem} For any $\delta >0$ the following bounds hold (where the constant $C_\delta \uparrow \infty$ as $\delta \downarrow 0$):
\begin{align} \label{bd:1}
\| \p_{\psi} Q^{(k)} \la \psi\ra^m \|_{L^2}^2& \le  C_{\delta} \| F^{(k)} \la \psi\ra^m \|_{L^2}^2 +C_{\delta} \| \p_{\psi}^2 G^{(k)} \la \psi\ra^m \|_{L^2}^2 + C_{\delta} \| b \|_{H^{k+1}_s}^2  \\
&\quad+ \delta \| Q^{(k)} \la \psi\ra^m \|_{L^2}^2 + C_\delta \bold{1}_{k \ge 1} \sum_{k' = 0}^{k-1} \| \p_\psi Q^{(k')} \la \psi\ra^m \|_{L^2}^2 + \bold{1}_{m \ge 1}\| Q^{(k)} \la \psi\ra^{m-1} \|_{L^2}^2.
\end{align}
\end{lem}
\begin{proof}
We multiply \eqref{hghg:k} by $\mathring{Q}^{(k)}\la \psi\ra^{2m}$ and integrate by parts to get the identity
\[
\bega 
\frac{\p_s}{2} \int_{\mathbb{R}_+} |\mathring{Q}^{(k)}|^2 \la\psi\ra^{2m}\rmd \psi &+ \cc q_e (s) \int_{\mathbb{R}_+} |\partial_{\psi}\mathring{Q}^{(k)}|^2 \la\psi\ra^{2m} \rmd \psi \\
&= \int (F^{(k)} + \p_{\psi}^2G^{(k)})\mathring{ Q}^{(k)} \la\psi\ra^{2m} \rmd \psi +  \int G^{(k)}_{\text{Lift}} \mathring{Q}^{(k)} \la\psi\ra^{2m} \rmd \psi \\
 &\qquad+ \frac{2m(2m-1)}{2} \cc q_e (s) \int |\mathring{Q}^{(k)}|^2 \la\psi\ra^{2m-2} \rmd \psi  \\
 &\qquad-  \int  \cc \bold{1}_{k \ge 1} \sum_{k' =0}^{k-1} \binom{k}{k'}\p_s^{k-k'} q_e  \p_\psi \mathring{Q}^{(k')} \p_\psi \mathring{Q}^{(k)} \la\psi\ra^{2m} \rmd \psi \\
 &\qquad -\int \w_0\bold{1}_{k\ge 1}\sum_{k'=0}^{k-1}\binom{k}{k'}\pt_s^{k-k'}q_e\pt_\psi \mathring{Q}^{(k')}\mathring{Q}^{(k)} 2m\la \psi\ra^{2m-1}.
 \enda 
\]
We now integrate in $s \in \mathbb{T}$, and the $\p_s$ term drops out due to periodicity. This implies 
\beq
\bega 
\lw\|\pt_\psi \mathring Q^{(k)}\la \psi\ra^m\rw\|_{L^2}^2&\lesssim \lw|\int\int  (F^{(k)} + \p_{\psi}^2 G^{(k)}) \mathring{Q}^{(k)} \la\psi\ra^{2m} \rmd \psi ds
\rw|\\
&\quad+\lw|\int\int G^{(k)}_{\text{Lift}} \mathring{Q}^{(k)} \la\psi\ra^{2m} \rmd \psi ds
\rw|\\
&\quad+\lw\|\mathring{Q}^{(k)}\la\psi\ra^{m-1}
\rw\|_{L^2}^2\\
&\quad+\lw|\int  \int  \cc \bold{1}_{k \ge 1} \sum_{k' =0}^{k-1} \binom{k}{k'}\p_s^{k-k'} q_e  \p_\psi \mathring{Q}^{(k')} \p_\psi \mathring{Q}^{(k)} \la\psi\ra^{2m} \rmd \psi ds\rw|\\
&\quad+\lw|\int \int \w_0\bold{1}_{k\ge 1}\sum_{k'=0}^{k-1}\binom{k}{k'}\pt_s^{k-k'}q_e\pt_\psi\mathring{ Q}^{(k')}\mathring{Q}^{(k)}\la \psi\ra^{2m-1}\rmd \psi \rmd s\rw|.
\enda
\eeq
This implies 
\[\bega 
\lw\|\pt_\psi \mathring Q^{(k)}\la \psi\ra^m\rw\|_{L^2}^2&\le \delta \lw\|\mathring{Q}^{(k)}\la\psi\ra^m
\rw\|_{L^2}^2+C_\delta \lw\|F^{(k)}\la \psi\ra^m
\rw\|_{L^2}^2+C_\delta\lw\|\pt_\psi^2 G^{(k)}\la\psi\ra^m
\rw\|_{L^2}^2\\
&\quad+C_\delta \|b\|_{H^{k+1}_s}+ C_0\bold{1}_{\{m \ge 1\}} \lw\|\mathring{Q}^{(k)}\la \psi\ra^{m-1}\rw\|_{L^2}^2\\
&\quad+C_\delta\bold{1}_{k \ge 1}\sum_{k'=0}^{k-1}\lw\|\pt_\psi \mathring{Q}^{(k')}\la \psi\ra^m
\rw\|_{L^2}^2+\delta\lw\|\pt_\psi\mathring{Q}^{(k)}\la\psi\ra^m,
\rw\|_{L^2}^2
\enda 
\]
where $\delta>0$ is small and $C_\delta\sim \delta^{-1}$.
The result follows immediately, using the fact that 
\[
\mathring{Q}^{(k)}=Q^{(k)}-e^{-\psi}b^{(k)}(s).
\]
  This concludes the proof of the lemma. 
\end{proof}

\begin{lem} Let $k \ge 0, m \ge 0$. The solution $Q^{(k)}$ to \eqref{real:2} satisfies the following estimate: 
\begin{align} 
\| \p_s Q^{(k)} \la\psi\ra^m \|_{L^2}^2 \lesssim &  \| F^{(k)} \la\psi\ra^m \|_{L^2}^2 + \| \p_{\psi}^2 G^{(k)} \la\psi\ra^m \|_{L^2}^2 + \| b \|_{H^{k+1}_s}^2 \\ \label{bd:2}
& +  \bold{1}_{m \ge 1} \| \p_\psi Q^{(k)} \la\psi\ra^{m-1} \|_{L^2}^2 + \bold{1}_{k \ge 1} \sum_{k' =0}^{k-1} \| \p_\psi^2 Q^{(k')} \la\psi\ra^m \|_{L^2}^2.
\end{align}
\end{lem}
\begin{proof} We multiply  \eqref{hghg:k}  by $\frac{1}{q_e (s)} \p_s \mathring{Q}^{(k)} \la\psi\ra^{2m}$ and integrate by parts to produce
\begin{align}
 &\int \frac{1}{q_e (s)} (F^{(k)} + \p_{\psi}^2 G^{(k)}) \p_s \mathring{Q}^{(k)} \la\psi\ra^{2m} \rmd \psi + \int \frac{1}{q_e (s)} G^{(k)}_{\text{Lift}} \p_s \mathring{Q}^{(k)} \la\psi\ra^{2m} \rmd \psi \\
 & \qquad + \int \frac{1}{q_e (s)}C_{\text{comm},k} \p_s \mathring{Q}^{(k)} \la\psi\ra^{2m} \rmd \psi \\
 = &  \int \frac{1}{q_e (s)} |\p_s \mathring{Q}^{(k)}|^2 \la\psi\ra^{2m} \rmd \psi + \cc \int \p_{\psi}^2 \mathring{Q}^{(k)} \p_s \mathring{Q}^{(k)} \la\psi\ra^{2m} \rmd \psi \\
 = &  \int \frac{1}{q_e (s)} |\p_s \mathring{Q}|^2 \la\psi\ra^{2m}\rmd \psi -  \cc \frac{\p_s}{2} \int |\p_\psi \mathring{Q}|^2 \la\psi\ra^{2m}\rmd \psi - \cc 2m \int \p_\psi \mathring{Q}^{(k)} \p_s \mathring{Q}^{(k)} \la\psi\ra^{2m-1} \rmd \psi .
\end{align}
Above, we have used the homogeneous boundary condition for $\mathring{Q}$ to integrate by parts. We now integrate over $s \in \mathbb{T}_L$ and use periodicity to eliminate the second term on the right-hand side above, which results in 
\[\bega
&\int\int\frac {1}{q_e(s)}\lw|\pt_s\mathring{Q}^{(k)}\rw|^2\la\psi\ra^{2m}\rmd \psi \rmd s\\
&\le C_\delta\lw\|F^{(k)}\la\psi\ra^m\rw\|_{L^2}^2+C_\delta\lw\|\pt_\psi^2 G^{(k)}\la\psi\ra^m\rw\|_{L^2}^2+C_\delta\lw\|G_{\text{Lift}}^{(k)}\la\psi\ra^m
\rw\|_{L^2}^2+C_\delta \lw\|C_{\text{comm},k} \la\psi\ra^m \rw\|_{L^2}^2\\
&\quad+\delta \lw\|\pt_s\mathring{Q}^{(k)}\la\psi\ra^m
\rw\|_{L^2}^2+C_\delta\lw\|\pt_\psi\mathring{Q}^{(k)}\la\psi\ra^{m-1}
\rw\|_{L^2}^2
\enda 
\]
Recalling \eqref{C-comm}, \eqref{G-Lift} and absorbing the last term on the right-hand side to the left, we get
\[\bega 
\lw\|\pt_s\mathring{Q}^{(k)}\la\psi\ra^m
\rw\|_{L^2}&\le C_\delta\lw(\lw\|F^{(k)}\la\psi\ra^m\rw\|_{L^2}^2+\lw\|\pt_\psi^2 G^{(k)}\la\psi\ra^m\rw\|_{L^2}^2+\|b\|_{H^{k+1}_s}^2\rw)\\
&\quad+C_\delta \bold{1}_{k \ge 1} \sum_{k' =0}^{k-1}  \lw\|\pt_\psi^2\mathring{Q}^{(k')}\la\psi\ra^m\rw\|_{L^2}^2
\enda
\]
We conclude the proof of the lemma, upon using the fact that $
\mathring{Q}^{(k)}=Q^{(k)}-e^{-\psi}b^{(k)}(s).$
 \end{proof}

We now need to estimate the zero mode of $Q^{(k)}$. Clearly, this is nontrivial only for $k = 0$ (when $k\ge 1$ there is no zero mode). 
\begin{lem} The zero mode, $Q^{(=0)}$, to the solution of \eqref{real:1}, satisfies the following bound: 
\begin{align} \label{bd:3}
\|  Q^{(= 0)} \la\psi\ra^m\|_{L^2_\psi}^2 \lesssim \  &\|  Q^{(\neq 0)} \la\psi\ra^m \|_{L^2_\psi}^2 + \| F \langle \psi \rangle^{m + 4} \|_{L^2}^2 + \| G  \la\psi\ra^m\|_{L^2}^2.
\end{align}
\end{lem}
\begin{proof}We integrate equation \eqref{real:1} to generate the identity for each $\psi \in \mathbb{R}_+$:
\begin{align}
\p_{\psi}^2 \int_0^L  q_e (s) Q(s, \psi) \rmd s = - \frac{1}{\cc}\int_0^L (F + \p_\psi^2 G) \rmd s ,
\end{align}
after which we integrate twice from $\infty$ to get 
\begin{align}
\int_0^L  q_e (s) Q(s, \psi) \rmd s = &- \frac{1}{\cc} \int_{\psi}^\infty \int_{\psi'}^\infty  \int_0^L(F + \p_\psi^2 G )  \rmd s \rmd \psi '' \psi' \\
= & - \frac{1}{\cc} \int_{\psi}^\infty \int_{\psi'}^\infty  \int_0^L F  \rmd s \rmd \psi '' \psi' - \frac{1}{\cc} \int_0^L G  \rmd s.
\end{align}
We now separate out the left-hand side 
\begin{align}
\int_0^L  q_e (s) Q(s, \psi) \rmd s &= \langle q_e  \rangle Q^{(= 0)}(\psi) + \int_0^L  (q_e (s) - \langle q_e  \rangle) Q(s, \psi) \rmd s \\
&=\la q_e \ra Q^{(=0)}(\psi)+\int_0^L \lw(q_e -\la q_e \ra
\rw)\lw(Q^{(=0)}(\psi)+Q^{(\neq 0)}(s,\psi)
\rw)\rmd s\\
&=\la q_e \ra Q^{(=0)}(\psi)+\int_0^L q_e ^{(\neq 0)}(s)Q^{(\neq 0)}(s,\psi)\rmd s
\end{align}
This implies 
\[\bega 
Q^{(=0)}(\psi)&=-\frac 1 {\w_0\la q_e \ra }\int_\psi^\infty \int_{\psi'}^\infty \int_0^L F \rmd s\rmd \psi''\rmd \psi'-\frac 1 {\w_0\la q_e \ra }\int_0^L G\rmd s\\
&\quad-\frac 1 {\la q_e \ra}\int_0^L  q_e ^{(\neq 0)}(s)Q^{(\neq 0)}(s,\psi)\rmd s.
\enda 
\]
We therefore obtain 
\begin{align*}
\|  Q^{(= 0)} \la\psi\ra^m\|_{L^2_\psi}^2 \lesssim & \int_{\mathbb{R}_+} \Big( \int_0^L  q_e ^{(\neq 0)}(s) Q^{(\neq 0)}(s, \psi) \rmd s \Big)^2 \la\psi\ra^{2m} \rmd \psi \\
&+ \int_{\mathbb{R}_+} \la\psi\ra^{2m} \Big(  \int_{\psi}^\infty \int_{\psi'}^\infty  \int_0^L F \Big)^2 \rmd \psi + \| G^{(=0)} \la\psi\ra^m\|_{L^2_\psi}^2 \\
 =:  &\  \mathcal{I}_1 + \mathcal{I}_2 + \mathcal{I}_3 .
\end{align*}
Clearly, $\mathcal{I}_3$ is majorized by the last term on the right-hand side of \eqref{bd:3}. We will estimate the first term above, which we call $\mathcal{I}_1$.
Using H\"{o}lder's inequality, we get 
\begin{align*}
\mathcal{I}_1 := &  \int_{\mathbb{R}_+} \la\psi\ra^{2m} \Big( \int_0^L q_e ^{(\neq 0)}(s)Q^{(\neq 0)}(s, \psi) \rmd s \Big)^2 \rmd \psi \\
\lesssim &\  \| Q^{( \neq 0)}\la \psi\ra^m \|_{L^2}^2 .
\end{align*}
To estimate $\mathcal{I}_2$, we need to pay weights as follows using Cauchy-Schwartz: 
\begin{align}
\Big|\int_{\psi'}^\infty \int_0^L F\Big| \lesssim \| F \langle \psi \rangle^{m + 4} \|_{L^2} \langle \psi' \rangle^{-m-\frac72},
\end{align}
which therefore implies that $|\mathcal{I}_2| \lesssim \| F \langle \psi \rangle^{m + 4} \|_{L^2}$. This completes the proof.
\end{proof}

\begin{lem}Let $k \ge 0, m \ge 0$. The solution $Q^{(k)}$ to \eqref{real:2} satisfies the following estimate: 
\begin{align} 
\| \p_\psi^2 Q^{(k)} \la\psi\ra^m \|_{L^2}^2 \lesssim &\  \| \p_s Q^{(k)} \la\psi\ra^m \|_{L^2}^2 + \| F^{(k)} \la\psi\ra^m \|_{L^2}^2 + \| \p_\psi^2 G^{(k)} \la\psi\ra^m \|_{L^2}^2 \\ \label{bd:4}
&+ \bold{1}_{k \ge 1} \sum_{k' = 0}^{k-1} \| \p_\psi^2 Q^{(k')} \la\psi\ra^m \|_{L^2}^2.
\end{align}
\end{lem}
\begin{proof}We simply rearrange equation \eqref{real:2} and apply $L^2$ norm to both sides. 
\end{proof}
\textit{
Proof of Proposition \ref{pro:pro:1}}. Consolidating the bounds \eqref{bd:1}, \eqref{bd:2}, \eqref{bd:3}, \eqref{bd:4}, we proved \eqref{sam:ad:1}. 

\subsection{$Q_n$ estimates}

\begin{lem} \label{hyu:1}Assume \eqref{mainbd:1} is valid up to  index $n$ and \eqref{mainbd:2} is valid up to  index $n-1$. Then 
\begin{align}
\| Q_n \|_{X_{2, 50}} \le \eps^{0.99}.
\end{align}
\end{lem}
\begin{proof} For this bound, motivated by equation \eqref{approx}, we set 
\begin{align}
F &:= \left(1-  \tfrac{\cc_{n-1}  q_e }{\sqrt{\cc_{n-1} ^2 q_e ^2+ Q_{n-1} }}\right)\partial_s Q_{n-1}\\
G &:= 0 \\
b &:=  \ve \overline{\cc}_n q_e ^2(s) + 2\ve g(s)q_e (s) + \ve^2 g^2(s)
\end{align} 

According to \eqref{sam:ad:1}, we fix $k = 2, m = 50$, which results in 
\begin{align} \label{jkl:1}
\| Q_n \|_{X_{2,50}} \lesssim & \sum_{k' = 0}^{2} \| \p_s^{k'}  F \langle \psi \rangle^{54} \|_{L^2}  +\| b \|_{H^{3}_s}
\end{align}
We therefore estimate the two quantities appearing on the right-hand side above. To make notation simpler, we define
\begin{align} \label{def:big:U}
U_{n-1} := \frac{Q_{n-1}}{\cc_{n-1}^2 q_e ^2 + Q_{n-1}},
\end{align}  
then 
\[
F=\lw(1-\sqrt{1-U_{n-1}}\rw)\pt_s Q_{n-1}
\]
By a direct calculation, we have the following identities
\beq\label{UF-calc}
\bega 
\p_s U_{n-1} =  & \frac{\p_s Q_{n-1}}{\cc_{n-1}^2 q_e ^2 + Q_{n-1}} -  \frac{Q_{n-1}}{(\cc_{n-1}^2 q_e ^2 + Q_{n-1})^2} (\cc_{n-1}^2 \p_s q_e ^2 + \p_s Q_{n-1}), \\
\p_s^2 U_{n-1} =  & \frac{\p_s^2 Q_{n-1}}{\cc_{n-1}^2 q_e ^2 + Q_{n-1}} - 2\frac{\p_s Q_{n-1}}{(\cc_{n-1}^2 q_e ^2 + Q_{n-1})^2} (\cc_{n-1}^2 \p_s q_e ^2 + \p_s Q_{n-1}) \\
& +   2 \frac{Q_{n-1}}{(\cc_{n-1}^2 q_e ^2 + Q_{n-1})^3} (\cc_{n-1}^2 \p_s q_e ^2 + \p_s Q_{n-1})^2  \\
& -  \frac{Q_{n-1}}{(\cc_{n-1}^2 q_e ^2 + Q_{n-1})^2} (\cc_{n-1}^2 \p_s^2 q_e ^2 + \p_s^2 Q_{n-1}).\\
\enda\eeq
First we estimate $\|F\la\psi\ra^{54}\|_{L^2}$. We have 
\begin{align}
\|F \langle \psi \rangle^{54} \|_{L^2} &\lesssim   \| U_{n-1} \langle \psi \rangle^4 \|_{L^\infty} \| \p_s Q_{n-1} \langle \psi \rangle^{50} \|_{L^2} \\
&\lesssim \frac{1}{1 - \| Q_{n-1} \|_{L^\infty}}  \| Q_{n-1} \langle \psi \rangle^4 \|_{L^\infty} \| \p_s Q_{n-1} \langle \psi \rangle^{50} \|_{L^2} \\
&\lesssim  \frac{1}{1 - \ve^{0.99}}\| Q_{n-1} \langle \psi \rangle^4 \|_{X_{1,4}}\| \p_s Q_{n-1} \langle \psi \rangle^{50} \|_{L^2} \\
&\lesssim  \| Q_{n-1} \|_{X_{1,4}} \| Q_{n-1} \|_{X_{0,50}} \\
&\lesssim \  \ve^{0.99} \ve^{0.99}.
\end{align} 
We now show that $\lw\|\p_s F\la\psi\ra^{54}\rw\|_{L^2}\lesssim \eps^{0.99}\eps^{0.99}$. We have 
\[\bega 
\p_s F = & (1 - \sqrt{1-U_{n-1}}) \p_s^2 Q_{n-1} + \frac12 (1 - U_{n-1})^{-\frac12} \p_s U_{n-1} \p_s Q_{n-1} \\
=: & A_{1} + A_2 \\
\enda 
\]
We first bound $A_1$.
We have 
\begin{align}
\| A_1 \langle \psi \rangle^{54} \|_{L^2} \lesssim  & \| U_{n-1} \langle \psi \rangle^4 \|_{L^\infty} \| \p_s^2 Q_{n-1} \langle \psi \rangle^{50} \|_{L^2} \\
\lesssim & \frac{1}{1 - \| Q_{n-1} \|_{L^\infty}}  \| Q_{n-1} \langle \psi \rangle^4 \|_{L^\infty} \| \p_s^2 Q_{n-1} \langle \psi \rangle^{50} \|_{L^2} \\
\lesssim & \frac{1}{1 - \ve^{0.99}}\| Q_{n-1}  \|_{X_{1,4}}\| \p_s^2 Q_{n-1} \langle \psi \rangle^{50} \|_{L^2} \\
\lesssim & \| Q_{n-1} \|_{X_{1,4}} \| Q_{n-1} \|_{X_{1,50}} \\
\lesssim &\  \ve^{0.99} \ve^{0.99}.
\end{align} 
We now estimate $A_2$. We have 
\begin{align}
\| A_2 \langle \psi \rangle^{54} \|_{L^2} \lesssim  & \| \p_s U_{n-1} \langle \psi \rangle^{4} \|_{L^\infty} \| \p_s Q_{n-1} \langle \psi \rangle^{50} \|_{L^2} \\
\lesssim & \| Q_{n-1} \|_{X_{2,4}} \| Q_{n-1} \|_{X_{0,50}} \\
\lesssim &\  \ve^{0.99} \ve^{0.99}.
\end{align}
We now show that 
\[
\lw\|\pt_s^2 F\la\psi\ra^{54}
\rw\|_{L^2}\lesssim \eps^{0.99}\eps^{0.99}.
\]
By a direct calculation, we get \[\bega
\p_s^2 F &=  (1 - \sqrt{1-U_{n-1}}) \p_s^3 Q_{n-1} + (1 - U_{n-1})^{-\frac12} \p_s U_{n-1} \p_s^2 Q_{n-1} \\
&+ (1 - U_{n-1})^{-\frac12} \p_s^2 U_{n-1} \p_s Q_{n-1} + (1 - U_{n-1})^{-\frac32}|\p_s U_{n-1} |^2 \p_s^2 Q_{n-1} \\
&= : B_1 + B_2 + B_3 + B_4. 
\enda\]
We first establish the following bounds on the auxiliary quantities $U_{n-1}$. We have 
\begin{align}
\| \p_s U_{n-1} \langle \psi \rangle^{m} \|_{L^\infty}& \lesssim  \| \p_s Q_{n-1} \langle \psi \rangle^{m} \|_{L^\infty} + \|  Q_{n-1} \langle \psi \rangle^{m} \|_{L^\infty} (1 + \| \p_s Q_{n-1} \|_{L^\infty}) \\
&\lesssim  \|  Q_{n-1}\|_{X_{2,m}} + \|  Q_{n-1} \|_{X_{2,m}} (1 + \| Q_{n-1} \|_{X_{2,0}}) \\
&\lesssim  \|  Q_{n-1}\|_{X_{2,m}}.
\end{align}
Similarly, we have
\begin{align}
\| \p_s^2 U_{n-1} \langle \psi \rangle^{m} \|_{L^2}& \lesssim \| \p_s^2 Q_{n-1} \langle \psi \rangle^{m} \|_{L^2} + \| \p_s Q_{n-1} \langle \psi \rangle^{m} \|_{L^2} (1 + \| \p_s Q_{n-1} \|_{L^\infty}) \\
& \quad+ \| Q_{n-1} \langle \psi \rangle^m \|_{L^2} (1 + \| \p_s Q_{n-1} \|_{L^\infty})^2 \\
& \quad+ \| Q_{n-1} \langle \psi \rangle^m \|_{L^\infty} (1 + \|\p_s^2 Q_{n-1} \|_{L^2}) \\
&\lesssim  \| Q_{n-1} \|_{X_{2,m}}
\end{align}
We can now estimate of $\lw\|\p_s^2 F\la\psi\ra^{54}\rw\|_{L^2}$. We first bound $B_1$. We have  
\begin{align}
\| B_1 \langle \psi \rangle^{54} \|_{L^2} &\lesssim  \| U_{n-1} \langle \psi \rangle^4 \|_{L^\infty} \| \p_s^3 Q_{n-1} \langle \psi \rangle^{50} \|_{L^2} \\
&\lesssim  \frac{1}{1 - \| Q_{n-1} \|_{L^\infty}}  \| Q_{n-1} \langle \psi \rangle^4 \|_{L^\infty} \| \p_s^3 Q_{n-1} \langle \psi \rangle^{50} \|_{L^2} \\
&\lesssim  \frac{1}{1 - \ve^{0.99}} \| Q_{n-1} \|_{X_{1,4}} \| Q_{n-1} \|_{X_{2,50}} \\
&\lesssim  \| Q_{n-1} \|_{X_{1,4}} \| Q_{n-1} \|_{X_{2,50}} \\
&\lesssim \ \ve^{0.99} \ve^{0.99}.
\end{align} 
We next move to $B_2$, for which we have 
\begin{align}
\| B_2 \langle \psi \rangle^{54} \|_{L^2} \lesssim  &\| \p_s U_{n-1} \langle \psi \rangle^{4} \|_{L^\infty} \| \p_s^2 Q_{n-1} \langle \psi \rangle^{50} \|_{L^2} \\
\lesssim & \| Q_{n-1} \|_{X_{2,4}} \| Q_{n-1} \|_{X_{1,50}} \\
\lesssim &\  \ve^{0.99} \ve^{0.99}.
\end{align}
As for $B_3$, we have 
\begin{align}
\| B_3 \langle \psi \rangle^{54} \|_{L^2} \lesssim  &\| \p_s^2 U_{n-1} \langle \psi \rangle^{50} \|_{L^2} \| \p_s Q_{n-1} \langle \psi \rangle^{4} \|_{L^\infty} \\
 \lesssim  &\| \p_s^2 U_{n-1} \langle \psi \rangle^{50} \|_{L^2} \| \p_s Q_{n-1} \langle \psi \rangle^{4} \|_{H^1_s H^1_\psi} \\
\lesssim & \| Q_{n-1} \|_{X_{2,50}} \| Q_{n-1} \|_{X_{2,4}} \\
\lesssim & \ \ve^{0.99} \ve^{0.99}.
\end{align}
We finally conclude with an estimate on $B_4$, for which we have 
\begin{align}
\| B_4 \langle \psi \rangle^{54} \|_{L^2} \lesssim  & \| \p_s U_{n-1} \langle \psi \rangle^2 \|_{L^\infty}^2 \| \p_s^2 Q_{n-1} \langle \psi \rangle^{50} \|_{L^2} \\
\lesssim & \| Q_{n-1} \|_{X_{2,2}}^2 \| Q_{n-1} \|_{X_{1,50}} \\
\lesssim &\ \ve^{0.99} \ve^{0.99} \ve^{0.99}.
\end{align}

To conclude the proof of lemma, we need to estimate the $H^3_s$ norm of $b$, 
\begin{align}
\| b \|_{H^3_s} \lesssim \ve |\cc_{n}| \| q_e  \|_{H^3_s} + \ve \| g \|_{H^3_s} \| q_e  \|_{H^3_s} + \ve^2 \| g \|_{H^3_s}^2 \lesssim \ve.
\end{align}
Therefore, according to \eqref{jkl:1}, the lemma is proven. 
\end{proof}

\subsection{$\Delta Q_n$ estimates}

Our main objective in this section is to close the final bootstrap bound, \eqref{mainbd:4}. We begin with a lemma which allows us to control our auxiliary quantity, $U_{n-1}$, introduced in \eqref{def:big:U}.  
\begin{lem} Let $0 \le m \le 50$. Assume \eqref{mainbd:1} - \eqref{mainbd:4} are valid until  the index $n - 1$. Assume \eqref{mainbd:1} - \eqref{mainbd:3} are valid until  the index $n$. The quantities $U_{n-1}, U_{n-2}$ satisfy:
\begin{align} \label{des:1}
\sum_{j = 0}^1 \| ( \p_s^j U_{n-1} - \p_s^j U_{n-2}) \langle \psi \rangle^m \|_{L^\infty} \lesssim &\| \Delta Q_{n-1} \|_{X_{2,m}} + \eps^{1.99} |\Delta \overline{\cc}_{n-1}| \\ \label{des:2}
\sum_{j = 0}^2 \| ( \p_s^j U_{n-1} - \p_s^j U_{n-2}) \langle \psi \rangle^m \|_{L^2} \lesssim &\| \Delta Q_{n-1} \|_{X_{2,m}} + \eps^{1.99} |\Delta \overline{\cc}_{n-1}|
\end{align}
\end{lem}
\begin{proof} Recalling \eqref{def:big:U}, 
we have 
\begin{align}
U_{n-1} - U_{n-2} = & \Big(\frac{Q_{n-1}}{\cc_{n-1}^2 q_e ^2 + Q_{n-1}} - \frac{Q_{n-2}}{\cc_{n-1}^2 q_e ^2 + Q_{n-1}}\Big)\\
&\qquad  + Q_{n-2} \Big( \frac{1}{\cc_{n-1}^2 q_e ^2 + Q_{n-1}} - \frac{1}{\cc_{n-2}^2 q_e ^2 + Q_{n-2}}  \Big) \\
= & \  \frac{\Delta Q_{n-1}}{\cc_{n-1}^2 q_e ^2 + Q_{n-1}} + \frac{Q_{n-2}}{(\cc_{n-1}^2 q_e ^2 + Q_{n-1})(\cc_{n-2}^2 q_e ^2 + Q_{n-2})} \Delta Q_{n-1} \\
&\quad + \frac{Q_{n-2}}{(\cc_{n-1}^2 q_e ^2 + Q_{n-1})(\cc_{n-2}^2 q_e ^2 + Q_{n-2})} q_e ^2 ( \cc_{n-2}^2 - \cc_{n-1}^2) \\ \label{calama:1}
= &\  \alpha \Delta Q_{n-1} + \gamma  \ve |\Delta \overline{\cc}_{n-1}|,
\end{align}
where the coefficients are defined by 
\begin{align}
\alpha := &\frac{1}{\cc_{n-1}^2 q_e ^2 + Q_{n-1}} + \frac{Q_{n-2}}{(\cc_{n-1}^2 q_e ^2 + Q_{n-1})(\cc_{n-2}^2 q_e ^2 + Q_{n-2})} \\
\gamma := & \frac{Q_{n-2}}{(\cc_{n-1}^2 q_e ^2 + Q_{n-1})(\cc_{n-2}^2 q_e ^2 + Q_{n-2})} q_e ^2.
\end{align}
According to our bootstraps, we claim the following bounds. There exists a decomposition of $\p_s^2 \alpha = \alpha_A + \alpha_B$ such that
\begin{align} \label{atacama:1}
\| \alpha \|_{L^\infty} + \| \p_s \alpha \|_{L^\infty} &\lesssim  1\\ \label{atacama:2}
\| \alpha_A \|_{L^\infty} + \| \alpha_B \|_{L^2}& \lesssim  1 \\ \label{atacama:3}
\| \gamma \langle \psi \rangle^m \|_{L^\infty} + \| \p_s \gamma \langle \psi \rangle^m \|_{L^\infty} &\lesssim \ve^{0.99} \\ \label{atacama:4}
 \| \p_s^2 \gamma \langle \psi \rangle^m \|_{L^2} &\lesssim  \ve^{0.99} \\ \label{at:4}
 \| \gamma \langle \psi \rangle^m \|_{L^2} + \| \p_s \gamma \langle \psi \rangle^m \|_{L^2}& \lesssim  \ve^{0.99}.
\end{align}
We will prove these bounds as follows. First, we define 
\begin{align}
\alpha^{(1)} := &\frac{1}{\cc_{n-1}^2 q_e ^2 + Q_{n-1}} = \frac{1}{D_{n-1}},\\
\alpha^{(2)} := &  \frac{Q_{n-2}}{(\cc_{n-1}^2 q_e ^2 + Q_{n-1})(\cc_{n-2}^2 q_e ^2 + Q_{n-2})} = \frac{Q_{n-2}}{D_{n-1} D_{n-2}}, \\
D_{n} := & \cc_{n}^2 q_e ^2 + Q_{n}. 
\end{align}
after which the following identities are valid: 
\begin{align} \label{atacama:01}
\alpha = \alpha^{(1)} + \alpha^{(2)}, \qquad \gamma = q_e ^2 \alpha^{(2)}. 
\end{align}
We will henceforth prove the following bounds. We claim there exists a decomposition of $\p_s^2 \alpha^{(1)} = \alpha^{(1)}_A + \alpha^{(1)}_B$, where 
\begin{align} \label{atacama:5}
\| \alpha^{(1)} \|_{L^\infty} + \| \p_s \alpha^{(1)} \|_{L^\infty} &\lesssim 1 \\  \label{atacama:6}
\| \alpha^{(1)}_A \|_{L^\infty} + \| \alpha^{(1)}_B \|_{L^2}& \lesssim  1 \\  \label{atacama:7}
\| \alpha^{(2)} \langle \psi \rangle^m \|_{L^\infty} + \| \p_s \alpha^{(2)} \langle \psi \rangle^m \|_{L^\infty}& \lesssim  \ve^{0.99} \\ \label{atacama:8}
 \| \p_s^2 \alpha^{(2)} \langle \psi \rangle^m \|_{L^2}& \lesssim \ve^{0.99} \\ \label{aat:4}
 \| \alpha^{(2)} \langle \psi \rangle^m \|_{L^2} + \| \p_s \alpha^{(2)} \langle \psi \rangle^m \|_{L^2} &\lesssim  \ve^{0.99},
\end{align}
upon which using \eqref{atacama:01}, we obtain \eqref{atacama:1}, \eqref{atacama:2}, \eqref{atacama:3}, \eqref{atacama:4}, and \eqref{at:4}.

\vspace{1 mm}

\noindent \textit{Proof of \eqref{atacama:5}:} Clearly, we have 
\begin{align}
\|  \alpha^{(1)} \|_{L^\infty} \lesssim & \frac{1}{\inf |D_{n-1}|} \lesssim \frac{1}{\cc_{n-1}^2 q_e ^2 - \| Q_{n-1} \|_{L^\infty} } \lesssim  \frac{1}{1- \ve^{0.99}} \lesssim 1. 
\end{align}
Next, we have the identity $\p_s \alpha^{(1)} = \frac{\p_s D_{n-1}}{D_{n-1}^2}$. Since we have already established a lower bound on $D_{n-1}$, it suffices to estimate $\p_s D_{n-1}$: 
\begin{align*}
\|  \p_s \alpha^{(1)} \|_{L^\infty} \lesssim & \| \p_s D_{n-1} \|_{L^\infty} \lesssim \| \cc_{n-1}^2 \p_s\{ q_e ^2 \} \|_{L^\infty} + \|\p_s  Q_{n-1} \|_{L^\infty} \lesssim  1 + \ve^{0.99} \lesssim 1, 
\end{align*}
where we have invoked the bootstraps \eqref{mainbd:1} and \eqref{mainbd:2}. This proves the bound \eqref{atacama:5}.
\vspace{1 mm}

\noindent \textit{Proof of \eqref{atacama:6}:} For this bound, we differentiate once more to find the identity 
\begin{align}
\p_s^2 \alpha^{(1)} = &  \frac{\p_s^2 D_{n-1}}{D_{n-1}^2} -2 \frac{|\p_s D_{n-1}|^2}{D_{n-1}^3} \\
= &[ \frac{1}{D_{n-1}^2}  \cc_{n-1}^2 \p_s^2 \{ q_e ^2 \}  + \frac{\p_sD_{n-1}}{D_{n-1}^3} \cc_{n-1}^2 \p_s \{ q_e ^2 \} ] + [ \frac{1}{D_{n-1}^2}\p_s^2 Q_{n-1}  + \frac{\p_sD_{n-1}}{D_{n-1}^3} \p_s Q_{n-1} ] \\
=:& \alpha^{(1)}_A + \alpha^{(1)}_B.
\end{align}
We estimate 
\begin{align}
\| \alpha^{(1)}_A \|_{L^\infty} \lesssim |\cc_{n-1}|^2 + \| \p_s D_{n-1} \|_{L^\infty} |\cc_{n-1}|^2 \lesssim 1, 
\end{align}
and 
\begin{align}
\| \alpha^{(1)}_B \|_{L^2} \lesssim \| \p_s^2 Q_{n-1} \|_{L^2} + \| \p_s D_{n-1} \|_{L^\infty} \| \p_s Q_{n-1} \|_{L^2} \lesssim \ve^{0.99}.
\end{align}
This proves the bound \eqref{atacama:6}.

\noindent \textit{Proof of \eqref{atacama:7}:} We turn now to the definition of $\alpha^{(2)}$. We will use freely the bounds $|D_{n-1}| + |D_{n-2}| \gtrsim 1$ and $\| \p_s D_{n-1} \|_{L^\infty} + \| \p_s D_{n-2} \|_{L^\infty} \lesssim 1$, which have already been established. First, we have 
\begin{align}
\| \alpha^{(2)} \langle \psi \rangle^m \|_{L^\infty} \lesssim \| Q_{n-2} \langle \psi \rangle^m \|_{L^\infty} \lesssim \| Q_{n-2} \|_{X_{2,m}} \lesssim \ve^{0.99}.
\end{align}
Next, we have the identity 
\begin{align} \label{santiago:1}
\p_s \alpha^{(2)} = \frac{\p_s Q_{n-2}}{D_{n-1} D_{n-2}} - \frac{Q_{n-2}\p_s \{ D_{n-1} D_{n-2} \}}{D_{n-1}^2 D_{n-2}^2}
\end{align}
from which we obtain 
\begin{align}
\| \p_s \alpha^{(2)} \langle \psi \rangle^m \|_{L^\infty} &\lesssim \| \p_s Q_{n-2} \langle \psi \rangle^m \|_{L^\infty} + \| Q_{n-2} \langle \psi \rangle^m \|_{L^\infty} ( \| D_{n-1} \|_{L^\infty} \| \p_s D_{n-2} \|_{L^\infty} \\
&\quad+ \| D_{n-2} \|_{L^\infty} \| \p_s D_{n-1} \|_{L^\infty}) \\
&\lesssim \| \ Q_{n-2}  \|_{X_{2,m}} + \| Q_{n-2}  \|_{X_{1,m}} ( \| D_{n-1} \|_{L^\infty} \| \p_s D_{n-2} \|_{L^\infty} \\
&\quad+ \| D_{n-2} \|_{L^\infty} \| \p_s D_{n-1} \|_{L^\infty}) \\
&\lesssim  \ve^{0.99}.
\end{align}
This proves the bound \eqref{atacama:7}.
\vspace{1 mm}

\noindent \textit{Proof of \eqref{atacama:8}:} We differentiate \eqref{santiago:1} again to obtain the identity 
\begin{align}
\p_s^2 \alpha^{(2)} = &  \frac{\p_s^2 Q_{n-2}}{D_{n-1} D_{n-2}} -2 \frac{\p_s Q_{n-2}\p_s \{ D_{n-1} D_{n-2} \}}{D_{n-1}^2 D_{n-2}^2} + 2\frac{Q_{n-2}|\p_s \{ D_{n-1} D_{n-2} \}|^2 }{D_{n-1}^3 D_{n-2}^3},
\end{align}
after which we obtain the bound 
\begin{align}
\| \p_s^2 \alpha^{(2)} \langle \psi \rangle^m \|_{L^2} \lesssim &\| \p_s^2 Q_{n-2} \langle \psi \rangle^m \|_{L^2} + \| \p_s Q_{n-2} \langle \psi \rangle^m \|_{L^2} \| \p_s \{ D_{n-1} D_{n-2} \} \|_{L^\infty} \\
& + \|Q_{n-2} \langle \psi \rangle^m \|_{L^2} \| \p_s \{ D_{n-1} D_{n-2} \} \|_{L^\infty}^2 \\
\lesssim & \| Q_{n-2} \|_{X_{2,m}} \\
\lesssim & \ve^{0.99}.
\end{align}
This proves the bound \eqref{atacama:8}. 

\vspace{1 mm}

\noindent \textit{Proof of \eqref{aat:4}:} We have 
\begin{align}
\| \alpha^{(2)} \langle \psi \rangle^m \|_{L^2} \lesssim \| Q_{n-2} \langle \psi \rangle^m \|_{L^2} \lesssim \| Q_{n-2} \|_{X_{2,m}} \lesssim \ve^{0.99},
\end{align}
and upon using \eqref{santiago:1}, we have 
\begin{align}
\| \p_s \alpha^{(2)} \langle \psi \rangle^m \|_{L^2} \lesssim &\| \p_s Q_{n-2} \langle \psi \rangle^m \|_{L^2} + \| Q_{n-2} \langle \psi \rangle^m \|_{L^2} ( \| D_{n-1} \|_{L^\infty} \| \p_s D_{n-2} \|_{L^\infty} \\
&+ \| D_{n-2} \|_{L^\infty} \| \p_s D_{n-1} \|_{L^\infty}) \\
\lesssim &\| \ Q_{n-2}  \|_{X_{2,m}} + \| Q_{n-2}  \|_{X_{1,m}} ( \| D_{n-1} \|_{L^\infty} \| \p_s D_{n-2} \|_{L^\infty} \\
&+ \| D_{n-2} \|_{L^\infty} \| \p_s D_{n-1} \|_{L^\infty}) \\
\lesssim & \ve^{0.99}.
\end{align}

We have therefore established \eqref{atacama:5} -- \eqref{aat:4} and hence \eqref{atacama:1} -- \eqref{at:4}. From here, the desired estimates, \eqref{des:1} -- \eqref{des:2}, follow from an application of the product rule applied to the identity \eqref{calama:1}. Indeed, we have:
\begin{align}
\| (U_{n-1} - U_{n-2} ) \langle \psi \rangle^m \|_{L^\infty} \lesssim & \| \alpha \|_{L^\infty} \| \Delta Q_{n-1} \langle \psi \rangle^m \|_{L^\infty}  +  \| \gamma \langle \psi \rangle^m \|_{L^\infty}  \ve |\Delta \overline{\cc}_{n-1}| \\
\lesssim & \| \Delta Q_{n-1}  \|_{X_{2,m}} + \ve^{1.99} |\Delta \overline{\cc}_{n-1}|,
\end{align}
where we have used the bounds \eqref{atacama:1} and \eqref{atacama:3}. In $L^2$, we similarly have 
\begin{align}
\| (U_{n-1} - U_{n-2} ) \langle \psi \rangle^m \|_{L^2} \lesssim & \| \alpha \|_{L^\infty} \| \Delta Q_{n-1} \langle \psi \rangle^m \|_{L^2}  +  \| \gamma \langle \psi \rangle^m \|_{L^2}  \ve |\Delta \overline{\cc}_{n-1}| \\
\lesssim & \| \Delta Q_{n-1}  \|_{X_{2,m}} + \ve^{1.99} |\Delta \overline{\cc}_{n-1}|,
\end{align}
where we have used the bounds \eqref{atacama:1} and \eqref{at:4}.

Next, we have upon differentiating \eqref{calama:1}, the identity 
\begin{align}
\p_s \{ U_{n-1} - U_{n-2} \} = \alpha \p_s \Delta Q_{n-1} + \Delta Q_{n-1} \p_s \alpha + \ve |\Delta \cc_{n-1}| \p_s \gamma,
\end{align}
after which we have the following $L^\infty$ bound: 
\begin{align}
\| \p_s \{ U_{n-1} - U_{n-2} \}  \langle \psi \rangle^m \|_{L^\infty} \lesssim & \| \alpha \|_{L^\infty} \| \p_s \Delta Q_{n-1} \langle \psi \rangle^m \|_{L^\infty} + \| \p_s \alpha \|_{L^\infty} \| \Delta Q_{n-1} \langle \psi \rangle^m \|_{L^\infty} \\
& + \ve \| \p_s \gamma \langle \psi \rangle^m \|_{L^\infty} |\Delta \cc_{n-1}| \\
\lesssim &  \| \Delta Q_{n-1}  \|_{X_{2,m}} + \ve^{1.99} |\Delta \cc_{n-1}|,
\end{align}
where we have invoked \eqref{atacama:1} and \eqref{atacama:3}. In $L^2$, we similarly have
\begin{align}
\| \p_s \{ U_{n-1} - U_{n-2} \}  \langle \psi \rangle^m \|_{L^2} \lesssim & \| \alpha \|_{L^\infty} \| \p_s \Delta Q_{n-1} \langle \psi \rangle^m \|_{L^2} + \| \p_s \alpha \|_{L^\infty} \| \Delta Q_{n-1} \langle \psi \rangle^m \|_{L^2} \\
& + \ve \| \p_s \gamma \langle \psi \rangle^m \|_{L^2} |\Delta \cc_{n-1}| \\
\lesssim &  \| \Delta Q_{n-1}  \|_{X_{2,m}} + \ve^{1.99} |\Delta \cc_{n-1}|,
\end{align}
where we have used the bounds \eqref{atacama:1} and \eqref{at:4}.

Differentiating \eqref{calama:1} twice in $s$, we obtain the identity 
\begin{align}
\p_s^2 \{ U_{n-1} - U_{n-2} \} = & \alpha \p_s^2 \Delta Q_{n-1} + \Delta Q_{n-1} \p_s^2 \alpha + 2  \p_s \Delta Q_{n-1} \p_s \alpha + \ve |\Delta \cc_{n-1}| \p_s^2 \gamma \\
= &  \alpha \p_s^2 \Delta Q_{n-1} + \Delta Q_{n-1} \alpha_A +  \Delta Q_{n-1} \alpha_B + 2  \p_s \Delta Q_{n-1} \p_s \alpha \\
& + \ve |\Delta \cc_{n-1}| \p_s^2 \gamma,
\end{align}
where we use the decomposition $\p_s^2 \alpha = \alpha_A + \alpha_B$. We now estimate the $L^2$ norm as follows: 
\begin{align}
\| \p_s^2 \{ U_{n-1} - U_{n-2} \} \langle \psi \rangle^m \|_{L^2} \lesssim & \| \alpha \|_{L^\infty} \|  \p_s^2 \Delta Q_{n-1} \langle \psi \rangle^m \|_{L^2} + \| \alpha_A \|_{L^\infty} \| \Delta Q_{n-1} \langle \psi \rangle^m \|_{L^2} \\
&+ \| \Delta Q_{n-1} \langle \psi \rangle^m \|_{L^\infty} \| \alpha_B \|_{L^2} + \| \p_s \Delta Q_{n-1} \langle \psi \rangle^m \|_{L^2} \| \p_s \alpha \|_{L^\infty} \\
& + \ve \| \p_s^2 \gamma \langle \psi \rangle^m \|_{L^2} |\Delta \cc_{n-1}| \\
\lesssim &\|  \p_s^2 \Delta Q_{n-1} \langle \psi \rangle^m \|_{L^2} +  \| \Delta Q_{n-1} \langle \psi \rangle^m \|_{L^2} + \| \Delta Q_{n-1} \langle \psi \rangle^m \|_{L^\infty} \\
&+ \| \p_s \Delta Q_{n-1} \langle \psi \rangle^m \|_{L^2}  + \ve^{1.99} |\Delta \cc_{n-1}| \\
\lesssim &  \| \Delta Q_{n-1}  \|_{X_{2,m}} + \ve^{1.99} |\Delta \cc_{n-1}|,
\end{align}
where we have used the bounds \eqref{atacama:1} -- \eqref{atacama:4}. We have therefore established the bounds \eqref{des:1} -- \eqref{des:2}, and this concludes the proof of the lemma. 
\end{proof}

\begin{lem} Assume \eqref{mainbd:1}--\eqref{mainbd:4} are valid until  the index $n - 1$. Assume \eqref{mainbd:1} -- \eqref{mainbd:3} are valid until  the index $n$. Then 
\begin{align} \label{desjk}
\| \Delta Q_n \|_{X_{2, 50}} \le &  \ve^{\frac34} \| \Delta Q_{n-1} \|_{X_{2,50}} + \ve^{\frac12} |\Delta \overline{\cc}_n| + \ve^{\frac32} |\Delta \overline{\cc}_{n-1}|.
\end{align}
\end{lem}
\begin{proof} For this estimate, motivated by \eqref{diffapprox}, we set 
\begin{align}
F := & \left(1-  \tfrac{\cc_{n-1}  q_e }{\sqrt{\cc_{n-1} ^2 q_e ^2+ Q_{n-1} }}\right)\partial_s Q_{n-1}-  \left(1-  \tfrac{\cc_{n-2}  q_e }{\sqrt{\cc_{n-2} ^2 q_e ^2+ Q_{n-2} }}\right)\partial_s Q_{n-2}, \\
G := & ( \cc_{n-1} -\cc_{n-2})  Q_{n-1} ,\\
b := &  \ve ( \overline{\cc}_{n}-  \overline{\cc}_{n-1}) q_e ^2(s).
\end{align} 
According to \eqref{sam:ad:1}, we fix $k = 2, m = 50$, which results in 
\begin{align} \label{jkl:1}
\| \Delta Q_n \|_{X_{2, 50}} \lesssim & \sum_{k' = 0}^{2} \| \p_s^{k'}  F \langle \psi \rangle^{54} \|_{L^2} +  \sum_{k' = 0}^{2} \sum_{j = 0}^2 \| \p_s^{k'} \p_\psi^j  G \langle \psi \rangle^{50} \|_{L^2}  +\| b \|_{H^{3}_s}.
\end{align}
We therefore estimate the two quantities appearing on the right-hand side above. We first address the term $F$, which we rewrite as follows 
\begin{align}
F = & (1 - \sqrt{1 - U_{n-1}}) \p_s Q_{n-1} - (1 - \sqrt{1 - U_{n-2}}) \p_s Q_{n-2} \\
= & (1 - \sqrt{1 - U_{n-1}}) \p_s \Delta Q_{n-1} + \p_s Q_{n-2}(\sqrt{1 - U_{n-1}} - \sqrt{1 - U_{n-2}}) := F_1 + F_2.
\end{align}
An identical calculation to the estimate of the forcing, $F$, in Lemma \ref{hyu:1} results in the bound 
\begin{align}
\| F_1 \langle \psi \rangle^{54} \|_{H^2_s L^2_\psi} \lesssim \ve^{0.99} \| \Delta Q_{n-1} \|_{X_{2,50}}.
\end{align} 
We develop the following identities 
\begin{align}
\p_s F_2 &=   \p_s^2 Q_{n-2}\left(\sqrt{1 - U_{n-1}} - \sqrt{1 - U_{n-2}}\right) - \frac12 \p_s Q_{n-2}\Big( \frac{\p_s U_{n-1}}{\sqrt{1 - U_{n-1}}} - \frac{\p_s U_{n-2}}{\sqrt{1-U_{n-2}}} \Big) \\
&=   \p_s^2 Q_{n-2}\left(\sqrt{1 - U_{n-1}} - \sqrt{1 - U_{n-2}}\right) - \frac12 \p_s Q_{n-2}\frac{\p_s U_{n-1} - \p_s U_{n-2}}{\sqrt{1 - U_{n-1}}}  \\
& \quad- \frac12 \p_s Q_{n-2} \Big( \frac{\p_s U_{n-2}}{\sqrt{1 - U_{n-1}}} - \frac{\p_s U_{n-2}}{\sqrt{1 - U_{n-2}}} \Big)\\
&=  C_1 + C_2 + C_3.
\end{align}

We estimate $\p_s F_2$ as follows. First, 
\begin{align*}
\| C_1 \langle \psi \rangle^{54} \|_{L^2}& \lesssim  \| (\sqrt{1 - U_{n-1}} - \sqrt{1 - U_{n-2}} ) \langle \psi \rangle^{4} \|_{L^\infty} \| \p_s^2 Q_{n-2} \langle \psi \rangle^{50} \|_{L^2} \\
&\lesssim  \| (U_{n-1} -  U_{n-2})  \langle \psi \rangle^{4} \|_{L^\infty} \| \p_s^2 Q_{n-2} \langle \psi \rangle^{50} \|_{L^2}  \\
&\lesssim  \| (U_{n-1} -  U_{n-2})  \langle \psi \rangle^{4} \|_{L^\infty} \| Q_{n-2}  \|_{X_{1,50}} \\
&\lesssim  \ \ve^{0.99}  \| (U_{n-1} -  U_{n-2})  \langle \psi \rangle^{4} \|_{L^\infty}.
\end{align*}
Next, to estimate $C_2$, we have 
\begin{align*}
\| C_2 \langle \psi \rangle^{54} \|_{L^2}& \lesssim  \| \p_s Q_{n-2} \langle \psi \rangle^{50} \|_{L^\infty} \|( \p_s U_{n-1} - \p_s U_{n-2}) \langle \psi \rangle^{4} \|_{L^2} \\
&\lesssim  \| \p_s Q_{n-2}  \|_{X_{2,50}} \|( \p_s U_{n-1} - \p_s U_{n-2}) \langle \psi \rangle^{4} \|_{L^2} \\
&\lesssim \  \ve^{0.99}\|( \p_s U_{n-1} - \p_s U_{n-2}) \langle \psi \rangle^{4} \|_{L^2}.
\end{align*}
Finally, to estimate $C_3$, we have 
\begin{align*}
\| C_3 \langle \psi \rangle^{54} \|_{L^2} \lesssim &\lw \| \p_s Q_{n-2} \langle \psi \rangle^{50} \rw\|_{L^2} \| \p_s U_{n-2} \langle \psi \rangle^{4} \|_{L^\infty} \| U_{n-1} - U_{n-2} \|_{L^\infty} \\
\lesssim & \| Q_{n-2} \|_{X_{0,50}} \| Q_{n-2} \|_{X_{2,4}}\| U_{n-1} - U_{n-2} \|_{L^\infty} \\
\lesssim &\  \ve^{0.99} \ve^{0.99} \| U_{n-1} - U_{n-2} \|_{L^\infty}.
\end{align*}

We now move to the second derivative, $\p_s^2 F_2$, which we will treat as follows:  
\begin{align}
\p_s^2 F_2 = \p_s C_1 + \p_s C_2 + \p_s C_3.
\end{align}
We have 
\begin{align}
\p_s C_1 = & \p_s^3 Q_{n-2}(\sqrt{1 - U_{n-1}} - \sqrt{1 - U_{n-2}}) -\frac12 \p_s^2 Q_{n-2} \Big( \frac{\p_s U_{n-1}}{\sqrt{1- U_{n-1}}} - \frac{\p_s U_{n-2}}{\sqrt{1 - U_{n-2}}} \Big) \\
= &  \p_s^3 Q_{n-2}(\sqrt{1 - U_{n-1}} - \sqrt{1 - U_{n-2}})-\frac12 \p_s^2 Q_{n-2} \Big( \frac{\p_s U_{n-1} - \p_s U_{n-2}}{\sqrt{1- U_{n-1}}} \Big) \\
& - \frac12 \p_s^2 Q_{n-2} \Big( \frac{\p_s U_{n-2}}{\sqrt{1 - U_{n-1}}} - \frac{\p_s U_{n-2}}{\sqrt{1 - U_{n-2}}}\Big) \\
= :& C_{1,1} + C_{1,2} + C_{1,3}.
\end{align}
First, we estimate 
\begin{align*}
\| C_{1,1} \langle \psi \rangle^{54} \|_{L^2} &\lesssim\lw  \| (\sqrt{1 - U_{n-1}} - \sqrt{1 - U_{n-2}} ) \langle \psi \rangle^{4}\rw \|_{L^\infty} \lw\| \p_s^3 Q_{n-2} \langle \psi \rangle^{50} \rw\|_{L^2} \\
&\lesssim   \lw\| (U_{n-1} -  U_{n-2})  \langle \psi \rangle^{4} \rw\|_{L^\infty} \| \p_s^3 Q_{n-2} \langle \psi \rangle^{50} \|_{L^2}  \\
&\lesssim  \lw \| (U_{n-1} -  U_{n-2})  \langle \psi \rangle^{4}\rw \|_{L^\infty} \| Q_{n-2}  \|_{X_{2,50}} \\
&\lesssim  \ \ve^{0.99}  \| (U_{n-1} -  U_{n-2})  \langle \psi \rangle^{4} \|_{L^\infty}.
\end{align*}
Next, to estimate $C_{1,2}$, we have 
\begin{align*}
\| C_{1,2} \langle \psi \rangle^{54} \|_{L^2} &\lesssim  \| \p_s^2 Q_{n-2} \langle \psi \rangle^{50} \|_{L^2} \|( \p_s U_{n-1} - \p_s U_{n-2}) \langle \psi \rangle^{4} \|_{L^\infty} \\
&\lesssim  \| Q_{n-2}  \|_{X_{1,50}} \|( \p_s U_{n-1} - \p_s U_{n-2}) \langle \psi \rangle^{4} \|_{L^\infty} \\
&\lesssim  \ \ve^{0.99}\|( \p_s U_{n-1} - \p_s U_{n-2}) \langle \psi \rangle^{4} \|_{L^\infty}.
\end{align*}
Next, to estimate $C_{1,3}$, we have 
\begin{align*}
\| C_{1,3} \langle \psi \rangle^{54} \|_{L^2} & \lesssim  \| \p_s^2 Q_{n-2} \langle \psi \rangle^{50} \|_{L^2} \| \p_s U_{n-2} \langle \psi \rangle^{4} \|_{L^\infty} \| U_{n-1} - U_{n-2} \|_{L^\infty} \\
&\lesssim  \| Q_{n-2} \|_{X_{1,50}} \| Q_{n-2} \|_{X_{2,4}}\| U_{n-1} - U_{n-2} \|_{L^\infty} \\
&\lesssim \  \ve^{0.99} \ve^{0.99} \| U_{n-1} - U_{n-2} \|_{L^\infty}.
\end{align*}
We next move to the $\p_s C_2$ contributions, for which we record the identity 
\begin{align}
\p_s C_2 = &  - \frac12 \p_s^2 Q_{n-2}\frac{\p_s U_{n-1} - \p_s U_{n-2}}{\sqrt{1 - U_{n-1}}}   - \frac12 \p_s Q_{n-2}\frac{\p_s^2 U_{n-1} - \p_s^2 U_{n-2}}{\sqrt{1 - U_{n-1}}} \\
& + \frac14 \p_s Q_{n-2} \p_s U_{n-1} (1 - U_{n-1})^{-\frac32}(\p_s U_{n-1} - \p_s U_{n-2}) \\
=: & C_{2,1}+C_{2,2} + C_{2,3}.
\end{align} 
We first estimate $C_{2,1}$ for which we have 
\begin{align*}
\| C_{2,1} \langle \psi \rangle^{54} \|_{L^2} &\lesssim  \| \p_s^2 Q_{n-2} \langle \psi \rangle^{50} \|_{L^2} \| (\p_s U_{n-1} - \p_s U_{n-2}) \langle \psi \rangle^{4} \|_{L^\infty} \\
&\lesssim \| Q_{n-2} \|_{X_{1,50}} \| (\p_s U_{n-1} - \p_s U_{n-2}) \langle \psi \rangle^{4} \|_{L^\infty}\\
&\lesssim \  \ve^{0.99} \| (\p_s U_{n-1} - \p_s U_{n-2}) \langle \psi \rangle^{4} \|_{L^\infty}.
\end{align*}
Next, we have 
\begin{align*}
\| C_{2,2} \langle \psi \rangle^{54} \|_{L^2}& \lesssim  \| \p_s Q_{n-2} \langle \psi \rangle^{50} \|_{L^\infty} \|( \p_s^2 U_{n-1} - \p_s^2 U_{n-2}) \langle \psi \rangle^{4} \|_{L^2} \\
&\lesssim  \| \p_s Q_{n-2}  \|_{X_{2,50}} \|( \p_s^2 U_{n-1} - \p_s^2 U_{n-2}) \langle \psi \rangle^{4} \|_{L^2} \\
&\lesssim \  \ve^{0.99}\|( \p_s^2 U_{n-1} - \p_s^2 U_{n-2}) \langle \psi \rangle^{4} \|_{L^2}.
\end{align*}
Finally, we have the $C_{2,3}$ contribution for which we estimate
\begin{align*}
\| C_{2,3} \langle \psi \rangle^{54} \|_{L^2}& \lesssim  \| \p_s Q_{n-2} \langle \psi \rangle^{50} \|_{L^\infty} \| \p_s U_{n-1} \langle \psi \rangle^4 \|_{L^\infty} \| (\p_s U_{n-1} - \p_s U_{n-2}) \|_{L^2} \\
&\lesssim  \| Q_{n-2}  \|_{X_{2,50}} \|Q_{n-1} \|_{X_{2,4}} \| (\p_s U_{n-1} - \p_s U_{n-2}) \|_{L^2} \\
&\lesssim \  \ve^{0.99}  \ve^{0.99} \| (\p_s U_{n-1} - \p_s U_{n-2}) \|_{L^2}.
\end{align*}

We next compute $\p_s C_3$, which results in the following identity, 
\begin{align}
\p_s C_3 =&\   \frac12 \p_s^2 Q_{n-2} \Big( \frac{\p_s U_{n-2}}{\sqrt{1 - U_{n-1}}} - \frac{\p_s U_{n-2}}{\sqrt{1 - U_{n-2}}} \Big) \\
& +  \frac12 \p_s Q_{n-2} \Big( \frac{\p_s^2 U_{n-2}}{\sqrt{1 - U_{n-1}}} - \frac{\p_s^2 U_{n-2}}{\sqrt{1 - U_{n-2}}} \Big) \\
& \quad - \frac14 \p_s Q_{n-2} \p_s U_{n-2} \frac{1}{(1 - U_{n-1})^{\frac32}} (\p_s U_{n-1} - \p_s U_{n-2})  \\
&\quad \quad  - \frac14 \p_s Q_{n-2} |\p_s U_{n-2}|^2\left( \frac{1}{(1- U_{n-1})^{\frac32}}  - \frac{1}{(1- U_{n-2})^{\frac32}} \right) =: \sum_{i = 1}^4 C_{3,i}.
\end{align}
We estimate first $C_{3,1}$ as follows 
\begin{align}
\|  C_{3,1} \langle \psi \rangle^{54} \|_{L^2}& \lesssim \| \p_s^2 Q_{n-2} \langle \psi \rangle^{50} \|_{L^2} \| \p_s U_{n-2} \langle \psi \rangle^{4} \|_{L^\infty} \| U_{n-1} - U_{n-2} \|_{L^\infty} \\
&\lesssim \| Q_{n-2}\|_{X_{1,50}} \| Q_{n-2} \|_{X_{2,4}} \| U_{n-1} - U_{n-2} \|_{L^\infty} \\
&\lesssim \  \ve^{0.99} \ve^{0.99}  \| U_{n-1} - U_{n-2} \|_{L^\infty}.
\end{align}
Next, we estimate $C_{3,2}$ as follows 
\begin{align}
\|  C_{3,2} \langle \psi \rangle^{54} \|_{L^2} &\lesssim \| \p_s Q_{n-2} \langle \psi \rangle^{50} \|_{L^\infty} \| \p_s^2 U_{n-2} \langle \psi \rangle^4 \|_{L^2} \| U_{n-1} - U_{n-2} \|_{L^\infty} \\
&\lesssim  \|Q_{n-2}  \|_{X_{2,50}} \| \p_s^2 U_{n-2} \langle \psi \rangle^4 \|_{L^2} \| U_{n-1} - U_{n-2} \|_{L^\infty} \\
&\lesssim \|Q_{n-2}  \|_{X_{2,50}} \|  Q_{n-2}  \|_{X_{2,4}} \| U_{n-1} - U_{n-2} \|_{L^\infty} \\
&\lesssim \ \ve^{0.99} \ve^{0.99} \| U_{n-1} - U_{n-2} \|_{L^\infty}.
\end{align}
Next, we estimate $C_{3,3}$ as follows 
\begin{align}
\|  C_{3,3} \langle \psi \rangle^{54} \|_{L^2} &\lesssim  \| \p_s Q_{n-2} \langle \psi \rangle^{50} \|_{L^\infty} \| \p_s U_{n-2} \langle \psi \rangle^{4} \|_{L^\infty} \| \p_s U_{n-1} - \p_s U_{n-2} \|_{L^2} \\
&\lesssim  \|  Q_{n-2}  \|_{X_{2,50}} \| Q_{n-2}  \|_{X_{2,4}} \| \p_s U_{n-1} - \p_s U_{n-2} \|_{L^2} \\
&\lesssim \  \ve^{0.99}  \ve^{0.99} \| \p_s U_{n-1} - \p_s U_{n-2} \|_{L^2}.
\end{align}
Finally, we estimate $C_{3,4}$ as follows 
\begin{align}
\|  C_{3,4} \langle \psi \rangle^{54} \|_{L^2} &\lesssim  \| \p_s Q_{n-2} \langle \psi \rangle^{50} \|_{L^\infty} \| \p_s U_{n-2} \langle \psi \rangle^{2} \|_{L^\infty}^2 \|  U_{n-1} -  U_{n-2} \|_{L^2} \\
&\lesssim  \|  Q_{n-2}  \|_{X_{2,50}} \| Q_{n-2}  \|_{X_{2,2}}^2 \| U_{n-1} - U_{n-2} \|_{L^2} \\
&\lesssim \ \ve^{0.99}  \ve^{0.99} \ve^{0.99} \| U_{n-1} - U_{n-2} \|_{L^2}.
\end{align}
Now, upon invoking \eqref{des:1} -- \eqref{des:2}, the above estimates give 
\begin{align}\label{h20:1}
&\| F_2 \langle \psi \rangle^{54} \|_{H^2_s L^2_\psi} \\
&\lesssim \ve^{0.99} \Big( \sum_{j = 0}^1 \| ( \p_s^j U_{n-1} - \p_s^j U_{n-2}) \langle \psi \rangle^4 \|_{L^\infty} + \sum_{j = 0}^2 \| ( \p_s^j U_{n-1} - \p_s^j U_{n-2}) \langle \psi \rangle^4 \|_{L^2}\Big) \\ 
&\lesssim \  \ve^{0.99}\lw (\| \Delta Q_{n-1} \|_{X_{2,4}} + \eps^{1.99} |\Delta \overline{c}_{n-1}|\rw).
\end{align} 
Next, we clearly have 
\begin{align}
& \sum_{k' = 0}^{2} \sum_{j = 0}^2 \| \p_s^{k'} \p_\psi^j  G \langle \psi \rangle^{50} \|_{L^2}\\
 & \le  |\cc_{n-1} - \cc_{n-2}|  \sum_{k' = 0}^{2} \sum_{j = 0}^2\lw \| \p_s^{k'} \p_\psi^j Q_{n-1}   \langle \psi \rangle^{50} \rw\|_{L^2} \\
 &\lesssim  \frac{1}{\lw|\cc_{n-1} + \cc_{n-2}\rw|} |\cc_{n-1}^2 - \cc_{n-2}^2| \| Q_{n-1} \|_{X_{2,50}} \\
& \lesssim \ \ve \ve^{0.99}  |\overline{\cc}_{n-1} - \overline{\cc}_{n-2}|.
\end{align}
Finally, we have the boundary condition
\begin{align}
\| b \|_{H^3_s} \lesssim \ve \| q_e  \|_{H^3}^3 |\Delta \overline{\cc}_n|.
\end{align}
Consolidating all the above bounds with estimate \eqref{desjk} concludes the proof of the lemma. 
\end{proof}

\appendix

\section{Derivation of near-boundary Navier-Stokes equations} \label{derivationNS}

In this section, we give the detailed calculations for the Navier-Stokes equations claimed in section \ref{derivation}.
We recall the standard identities, which will be used in the next lemmas:
\[
\bega 
n'(s)=\gamma(s)\tau(s),&\qquad \tau'(s)=\gamma(s)n(s)
\enda
\]
We recall that the map
\[\bega
\{x\in M:\quad 0<{\rm dist}(x,\pt M)<\delta\}&\to \mathbb T_L\times (0,\delta)\\
x&\to (s,z)=\lw(s(x_1,x_2),z(x_1,x_2)
\rw)
\enda 
\] 
is a diffeomorphism. In this transformation, we have
\beq\label{iden-Tr}
\nabla_x s=\frac{\tau}{J},\qquad \nabla_x z=n(s).
\eeq
For a vector field $u:M\to \mathbb R^2$, we also have 
\beq\label{u-eq-Tr}
\bega 
&u_1=n_2u_\tau-\tau_2u_n=\tau_1u_\tau-\tau_2u_n,\\
&u_2=-n_1u_\tau+\tau_1u_n=\tau_2u_\tau+\tau_1u_n.
\enda
\eeq
\begin{lem}\label{lem1-Tr}
For any vector field $u:M\to \mathbb R^2$ and scalar function $f:M\to \mathbb R$  supported near the boundary $\partial M$ there holds
\[\bega 
&u\cdot \nabla f=\frac{u_\tau}{J}\pt_s f+u_n\pt_z f\\
\enda
\]
In particular, by choosing $u=\tau$ and $u=n$ respectively, there hold
\[\bega
\tau\cdot\nabla f=\frac{1}{J}\pt_s f,\qquad n\cdot\nabla f=\pt_z f\\
\enda
\]
\end{lem}
\begin{proof}
This follows by direct calculation. We have
\begin{align}
u\cdot \nabla f &=u\cdot\nabla_x f(s,z)=u_1\lw(\pt_sf\pt_{x_1}s+\pt_z f\pt_{x_1}z\rw)+u_2\lw(\pt_sf\pt_{x_2}s+\pt_z f\pt_{x_2}z\rw)\\
&=u_1\lw(\pt_s f\frac{\tau_1}{J}+\pt_z f n_1\rw)+u_2(\pt_s f\cdot\frac{\tau_2}{J}+\pt_z f n_2)\\
&=\frac{u\cdot \tau}{J}\pt_s f+(u\cdot n)\pt_z f.
\end{align}
\end{proof}

\begin{lem} The following identities holds for any given vector field $u:M\to \mathbb R^2$:
\[\bega 
(u\cdot \nabla u)\cdot \tau&=\lw(\frac{u_\tau}{J},u_n\rw)\cdot\nabla_{s,z}u_\tau-\frac{\gamma(s)}{J}u_\tau u_n ,\\
(u\cdot \nabla u)\cdot n&=\lw(\frac{u_\tau}{J},u_n\rw)\cdot\nabla_{s,z}u_n-\frac{\gamma(s)}{J}u_\tau^2,\\
\Delta u\cdot \tau&=\frac{1}{J}\pt_z(J\pt_zu_\tau)+\frac{1}{J}\pt_s\lw(\frac{1}{J}\pt_s u_\tau
\rw)-\frac{1}{J}\pt_s\lw(\frac{\gamma u_n}{J}
\rw)-\frac{\gamma}{J}\lw(\gamma u_\tau+\pt_s u_n
\rw),\\
\Delta u\cdot n&=\frac{1}{J}\pt_z(J\pt_zu_n)+\frac{1}{J}\pt_s\lw(\frac{1}{J}\pt_s u_n
\rw)-\frac{1}{J}\pt_s \lw(\frac{\gamma u_\tau}{J}
\rw)-\frac{\gamma}{J}\lw(\pt_s u_\tau-\gamma u_n \rw),\\
\nabla \cdot u&=\frac{1}{J}\lw(\pt_s u_\tau-\gamma(s)u_n
\rw)+\pt_z u_n,\\
\nabla^\perp \cdot u&=\pt_1u_2-\pt_2 u_1=\frac{\gamma}{J}u_\tau-\pt_zu_\tau+\pt_s u_n.
\enda 
\]
\end{lem}
\begin{proof}
We check the first, the third  and the fifth identities only, and the proofs for other identities are similar.
We have 
\[\bega 
(u\cdot\nabla u)\cdot\tau &=(u\cdot\nabla_x u_1)\tau_1+(u\cdot\nabla_x u_2)\tau_2\\
&=\lw(\frac{u_\tau}{J},u_n\rw) \cdot\nabla_{s,z} u_1 \tau_1+\lw(\frac{u_\tau}{J},u_n\rw)\cdot\nabla_{s,z} u_2 \tau_2\\
&=\lw(\frac{u_\tau}{J},u_n
\rw)\cdot\nabla_{s,z}u_\tau-\frac{u_\tau}{J} \lw(\tau_1'(s)u_1+\tau_2'(s)u_2\rw)
\enda 
\]
We note that 
\[\bega 
\tau_1'u_1+\tau_2'u_2&=\gamma n_1(\tau_1u_\tau-\tau_2u_n)+\gamma n_2(\tau_2u_\tau+\tau_1u_n)\\
&=\gamma(-\tau_2)(\tau_1u_\tau-\tau_2u_n)+\gamma \tau_1(\tau_2u_\tau+\tau_1u_n)\\
&=-\gamma\tau_1\tau_2u_\tau+\gamma\tau_2^2 u_n+\gamma \tau_1\tau_2u_\tau+\gamma \tau_1^2 u_n=\gamma u_n 
\enda
\]
Combining the above with the previous calculation, we obtain
\[
(u\cdot\nabla u)\cdot\tau=\lw(\frac{u_\tau}{J},u_n\rw)\cdot\nabla_{s,z}u_\tau-\frac{\gamma(s)}{J}u_\tau u_n \]
Now we show the third identity. We have 
\[\bega 
\Delta u\cdot \tau&=\Delta u_1\cdot \tau_1+\Delta u_2\cdot \tau_2=\sum_{i=1}^2\frac{1}{J}\lw(\pt_z(J\pt_z u_i)+\pt_s\lw(\frac{1}{J}\pt_s u_i\rw)
\rw)\tau_i\\
&=\frac{1}{J}\pt_z(J\pt_zu_\tau)+\frac{1}{J}\pt_\theta\lw(\frac{1}{J}\pt_\theta (u\cdot\tau)
\rw)-\frac{1}{J}\pt_s\lw(\frac{1}{J}u\cdot\pt_s \tau
\rw)-\sum_i \frac{1}{J}\pt_s  u_i\tau_i'\\
&=\frac{1}{J}\pt_z(J\pt_zu_\tau)+\frac{1}{J}\pt_\theta\lw(\frac{1}{J}\pt_s u_\tau
\rw)-\frac{1}{J}\pt_s\lw(\frac{\gamma u_n}{J}
\rw)-\frac{\gamma}{J}\lw(\gamma u_\tau+\pt_s u_n
\rw).
\enda 
\]
  For incompressibility, we find 
\[\bega 
\nabla \cdot u&=\pt_{x_1}u_1+\pt_{x_2}u_2=\sum_{i}\pt_{x_i}s \pt_s u_i+\pt_{x_i}z\pt_z u_i=\sum_i \frac{\tau_i}{J}\pt_s u_i+n_i\pt_z u_i\\
&=\frac{1}{J}\lw(\pt_s u_\tau-\gamma(s)u_n
\rw)+\pt_z u_n.
\enda 
\]
The proof is complete.
\end{proof}

\noindent \textbf{Acknowledgements.} 
The research of TDD was partially supported by the NSF
DMS-2106233 grant and  NSF CAREER award \#2235395. The research of SI was partially supported by NSF DMS-2306528.

\end{document}